\documentclass{imanum}
\usepackage{graphicx}
\usepackage{epstopdf}

\usepackage{amsfonts,amssymb,amsmath,empheq}
\usepackage{amsthm}
\usepackage{mathrsfs}
\usepackage{latexsym}
\usepackage{subfig}
\usepackage{cases}

\newtheorem{rem}{Remark}[section]
\newcommand{\beq}{\begin{equation}}
\newcommand{\eeq}{\end{equation}}

\newcommand{\beqa}{\begin{eqnarray}}
\newcommand{\eeqa}{\end{eqnarray}}
\newcommand{\pr}{\partial}

\usepackage{cancel}
\usepackage{xcolor}
\colorlet{colorYO}{red}
\colorlet{colorSJ}{blue}

\newcommand{\YOr}[2]{\textcolor{colorSJ}{\sout{#1}#2}}   

\begin{document}

\title{An efficient and energy decaying discontinuous Galerkin method for Maxwell's equations  for the Cole-Cole dispersive medium}
\shorttitle{Energy decay DG for Cole-Cole Maxwell's equations}

\author{Jiangming Xie}

\author{%
{\sc
Jiangming Xie\thanks{Email: xiejiangming@mail.tsinghua.edu.cn} } \\[2pt]
Department of Mathematical Sciences, Tsinghua University, Beijing 100084, P.R. China\\[6pt]
{\sc and}\\[6pt]
{\sc Maojun Li}\thanks{Email: limj@uestc.edu.cn}\\[2pt]
School of Mathematical Sciences, University of Electronic Science and Technology of China, Chengdu, Sichuan, 611731, P.R. China\\[6pt]
{\sc and}\\[6pt]
{\sc Miao-Jung Yvonne Ou}\thanks{Corresponding author. Email: mou@udel.edu}\\[2pt]
Department of Mathematical Sciences, University of Delaware, Newark, DE 19716, USA
}
\shortauthorlist{J. M. Xie \emph{et al.}}

\maketitle

\begin{abstract}
{In this work, we investigate the propagation of electromagnetic waves in the Cole-Cole dispersive medium  by using the discontinuous Galerkin (DG) method to solve  the coupled time-domain Maxwell's equations and polarization equation.
We define a new and   sharpened total energy function for the Cole-Cole model, which better describes the behaviors of the energy than  what is available in the current literature.
A major theme in the time-domain numerical modeling of this problem  has been tackling the difficulty of handling the nonlocal term  involved in the time-domain polarization equation. Based on the diffusive representation and the quadrature formula, we derive an approximate system, where the convolution kernel is replaced by a finite number of auxiliary  variables that satisfy  local-in-time ordinary differential equations.
To ensure the resulted approximate system is stable, a nonlinear constrained optimization numerical scheme is established  to determine the quadrature coefficients.  By a special choice of the numerical fluxes and projections, we  obtain {for the constant coefficient case } an optimal-order convergence result for the semi-discrete DG scheme. The temporal discretization is achieved by the standard two-step backward difference formula and a fast algorithm with linear  complexity  is constructed.
Numerical examples are provided for demonstrating the efficiency of the proposed algorithm, validating the theoretical results and  illustrating the behaviors of the energy.}
{Cole-Cole dispersive medium; Maxwell's equations; discontinuous Galerkin method; optimal-order convergence; fast algorithm; linear  complexity.
}
\end{abstract}

\section{Introduction}
\label{sec;introduction}
In electromagnetics, a medium with frequency dependent permittivity or permeability is called dispersive medium,  which is abundant in nature and has been studied  in various applications such as  biological tissue, ionosphere, water, soil,  plasma and  radar absorbing material; see \cite{GabrielBiological,HoekstraSoil,PostowBiological} for reference.
The propagation of electromagnetic waves in  dispersive  media is described by the time-domain Maxwell's equations coupled with the nonlocal polarization equations. Typical models that characterise the dispersion process include the Debye model \cite{Debye}, the Lorentz model \cite{Cassier2017,LuebbersLo}, the Havriliak-Negami model \cite{HavriliakNegami} and the Cole-Cole model \cite{CC1941,GarrappaDie}.
Numerical methods for these models include  the finite-difference time-domain (FDTD) method \cite{GuoBio,Rekanosfre}, the time-domain finite element (FETD) method \cite{Bokil,Jiao,LiChen}, the spectral time-domain method (STD) \cite{Zeng} and the time-domain discontinuous Galerkin (DGTD) method  \cite{BokilLiMaxwell,CockburnLiShu,HesthavenWarburton,HuangLi,LuZhang,LyuXu,WangXie,WangXieZhang}.

Since the Cole-Cole model can adequately capture the dispersive and dissipative behaviors of many biological materials, its numerical solution has received much attention. Intensive studies has been devoted to the  FDTD method.  Roughly speaking, for the FDTD method, two candidates  have  appeared  in the literature to deal with the nonlocal polarization equation. One of them approximates the fractional derivative by a sum of decaying exponentials \cite{bai2022second,Torres,  MrozowskiStuchly,RekanosPapadopoulos,SchusterLuebbers} and the other  approximates the induced polarization by a time convolution of the electric field \cite{CausleyJiang}.
In \cite{LiHuangLinCOLE}, the authors gave the first numerical analysis of the FETD method and defined the  energy function for the Cole-Cole model.
In \cite{WangHuang}, the authors used the STD method to study the Cole-Cole model by formulating the Cole-Cole model as a  second order partial integral-differential equation. 
In \cite{WangZhangZhang}, the authors  used the continuous Galerkin finite element method in time and DG method in space to solve the Cole-Cole model  \cite{WangZhangZhang}, where the sum-of-exponential approximation for the  fractional derivative was used to speed up the evolution.


Despite the  many aforementioned  works, for the Cole-Cole model, results on the energy analysis for the Cole-Cole model are scarce. As is shown in \cite{WangHuang} and \cite{bai2022second}, the energy formula used in the current literature \cite{LiHuangLinCOLE} is not monotonically decreasing with respect to time.
Since a monotonically decreasing total energy plays an important role in studying the stability of a dynamical system, the availability of a correct energy formula for the Cole-Cole model is of utmost importance. In contrast to the existing energy formula for the Cole-Cole model, the energy formula derived in Section \ref{governing equation} of this paper takes into account the diffusive part and is monotonically decreasing in time. More precisely, by reformulating 
the polarization equations with  the diffusive representation, we introduce the diffusive energy and  define a  new and sharped total energy function for the Cole-Cole model. We give a complete theoretical analysis to show the monotonic decreasing property of  the new defined total energy function. 
 
From the numerical analysis point of view, the most notable challenge in solving  the Cole-Cole model is the handling of the fractional derivative, which is a pseudo-differential operator, involved in the polarization equation.  Since the solutions depend on all time history,  a naive implementation will present a heavy computation burden especially for large number of time steps. To overcome this difficulty, we transform the nonlocal problem into local continuum   problems with the help of  the diffusive representation. Roughly speaking,  with the diffusive representation, the convolution kernel is represented by its Laplace transform; see \cite{BirkgauJa, blanc2014wave, DeschMiller,Diethelm08,HaddarMatignon,Lombardalpha,xie2019discontinuous,YuanAgrawal}  and the references therein for more details.
By discretizing the resulted infinite integral  with a quadrature formula,  the fractional derivative is replaced by a finite number of continuous auxiliary variables which  satisfy the local-in-time ordinary differential equations (ODEs) and hence can be handled easily. Since the  stability of the resulted approximate system strongly relies on the positivity of the  quadrature coefficients, a  nonlinear constrained optimization scheme is applied to determine the weights and abscissae.  
The spatial discretization is handled by using the DG method. 
By a special choice of numerical fluxes and projections, the semi-discrete DG scheme proposed in this paper has an optimal convergence rate of $\mathcal{O}(h^{k+1})$ with $k$ and $h$ being the degree of the polynomials and the spatial step-size, respectively. In comparison, the DG discretization used  in \cite{WangZhangZhang} for solving the same problem has a convergence rate of $\mathcal{O}\left(h^{k+0.5}\right)$.
 The standard two-step backward difference formula (BDF2) is used for the temporal discretization and   the overall complexity of the propose scheme is $\mathcal{O}(N\,\text{log}N)$ with N being the total time steps.

The remainder of this paper is organized as follows.  In Section \ref{governing equation}, we introduce the governing equations of  electromagnetic waves in the Cole-Cole dispersive medium and illustrate the energy property of this model.
In Section \ref{SpatialDG}, we present a DG scheme for spatial discretization  and prove the semi-discrete stability and estimate the semi-discrete errors.  In Section \ref{appmodel}, we derive the approximate system and the approximate energy function.  In Section \ref{optimization}, we  introduce the nonlinear optimization method to determine the quadrature coefficients. The dispersion analysis is given in  Section \ref{dispersion}. In Section \ref{numerical},  we provide numerical examples to demonstrate the efficiency  of our fast scheme and verify the theoretical results.
\section{Maxwell's equations in the Cole-Cole dispersive medium }
\label{governing equation}

In this section, we present the governing equations of electromagnetic wave propagation in dispersive media, and introduce the energy property of this model.

\subsection{Physical model}
For a Cole-Cole dispersive medium, the permittivity $\epsilon_r(\omega)$ in the frequency domain can be expressed as  \cite{CC1941}
$$
\epsilon_r(\omega)=\epsilon_{\infty}+\frac{\epsilon_s-\epsilon_{\infty}}
{1+\left(i\omega\tau_0\right)^\alpha},\quad 0<\alpha <1,  \quad \epsilon_s>\epsilon_{\infty},
$$
where $\omega$ is the  angular  frequency, $\epsilon_s,\,\epsilon_{\infty}$  the static permittivity and the infinite-frequency permittivity,  respectively, $\tau_0$ the relaxation time and $i=\sqrt{-1}$ is the imaginary unit. In terms of the Laplace transform $\mathcal{L}[\cdot]$, the   electric field ${E}$ and the induced polarization field ${P}$  are related by  
\begin{equation}\label{eq:freEP}
\mathcal{L}\left[{P}\right]({x})=\epsilon_0(\epsilon_r-\epsilon_{\infty})\mathcal{L}
\left[E\right]( {x})
=\frac{\epsilon_0(\epsilon_s-\epsilon_{\infty})}{1+(i\omega\tau_0)^{\alpha}}\mathcal{L}
\left[E\right]( {x}),\quad  {x}\in\Omega\subset\mathbb{R}^d,\ d=1,2,3.
\end{equation}
 For simplicity, we consider a one-dimensional Cole-Cole  model and note that the analysis in this paper can be easily generalized to higher dimensional models in a straightforward fashion. The time-domain governing equations for the one-dimensional Cole-Cole model consisting of the  Maxwell's equations and the inverse Laplace transform of (\ref{eq:freEP}):
\begin{eqnarray}
\label{eq:MaxColeColeH}
\mu_0 \frac{\pr H}{\pr t}&=& \frac{\partial E}{\partial x}+ {F}_1,\ \, \ \quad \ \ \ \quad\quad\, \text{in}\quad \Omega\times (0,T],\\
\label{eq:MaxColeColeE}
\epsilon_0\epsilon_{\infty}\frac{\pr E}{\pr t} &=&\frac{ \partial H}{\partial x} -\frac{\pr P}{\pr t}+{F}_2,\, \ \, \ \quad \text{in}\quad \Omega\times (0,T],\\
\label{eq:MaxColeColeP}
\tau_0^{\alpha}\frac{\pr^{\alpha}P}{\pr t^{\alpha}}+ P &=&\epsilon_0(\epsilon_s-\epsilon_{\infty})  E+F_3,\, \, \ \ \text{in}\quad \Omega\times (0,T],
\end{eqnarray}
where $H$ is the  magnetic field, $F_1,\,\ F_2,\,F_3 $ the external source functions,  $\Omega\subset\mathbb{R}^1$ a bounded interval, and $\epsilon_0,\,\mu_0$ represent the permittivity and  the permeability of the vacuum, respectively. Moreover, the  Caputo fractional derivative is defined as \cite{Podlubny}
\begin{equation}\label{def:Caputo}
\frac{\partial^{\alpha}P(x,t)}{\partial t^{\alpha}}:=\frac{t^{-\alpha}}{\Gamma(1-\alpha)}\ast \frac{\partial P(x,t)}{\partial t}=\frac{1}{\Gamma(1-\alpha)}\int_0^t\left(t-\tau\right)^{-\alpha} \frac{\partial P(x,\tau)}{\partial \tau}d\tau, \quad 0<\alpha<1,
\end{equation}
where $\ast$ indicates the temporal convolution and  $\Gamma(\cdot)$ is the gamma function
$$
\Gamma(\alpha)=\int_0^{\infty}e^{-\theta}\theta^{\alpha-1}d\theta,\quad\forall\, \theta\notin \mathbb{Z}.
$$
We consider the periodic boundary condition
and the following periodic initial conditions
\begin{equation}\label{eq:iniMax}
E (x,0)= E _0(x),\quad H(x,0)=H_0(x),\quad P(x,0)=P_0(x)=0.
\end{equation}

Note that in \eqref{eq:MaxColeColeH}, the spatial derivative has a different sign from the 3-D Maxwell equations. For more details, see \cite{McDonald2019Electrodynamics}.

\subsection{The diffusive representation }
The temporal convolution involved in the polarization equation  is hard to be implemented numerically. A direct discretization  method needs to store  the solutions history, which is expensive and inefficient.  In this section, we introduce the diffusive representation to transform the nonlocal integral into local problems. Suppose $P\in H^1(0,T; L^2(\Omega) )$. According to the Euler Gamma formula
$
\Gamma(1-\alpha)\Gamma(\alpha)=\frac{\pi}{\sin(\pi \alpha)},\,\alpha\notin \mathbb{Z}$,  \eqref{def:Caputo} becomes
\begin{equation}\label{eq:reDerivative}
\begin{aligned}
\frac{\partial^{\alpha}P(x,t)}{\partial t^{\alpha}}&=\frac{\sin(\pi\,\alpha)}{\pi}
\int_{0}^{\infty}\int_0^te^{-\theta}\left(\frac{\theta}{t-\tau}\right)^{\alpha}\,\frac{1}{\theta}
\,\frac{\partial P(x,\tau)}{\partial \tau}\,d\tau\,d\theta\\
&=\frac{\sin(\pi\,\alpha)}{\pi}\int_{0}^{\infty}\lambda^{\alpha-1}\int_{0}^te^{-(t-\tau)\lambda}\,\frac{\partial P(x,\tau)}{\partial \tau}\,d\tau\,d\lambda,
\end{aligned}
\end{equation}
where we have used $\theta=(t-\tau)\lambda$.

The analytic structure can be seen more clearly using the measure supported in $\mathbb{R}^+$ defined below and its Laplace transform,
\begin{eqnarray}
dM_\alpha(\lambda)=\mu_\alpha(\lambda)\, d\lambda \mbox{\quad with density\quad } \mu_\alpha(\lambda):=\frac{\sin \pi \alpha}{\pi}\, \lambda^{\alpha-1},\\ h_\alpha(t):=\int_0^\infty e^{-\lambda t} dM_\alpha(\lambda).\label{def:h_alpha}
\end{eqnarray}
Introducing the auxiliary variable
\begin{equation}\label{def:psidr}
 {\psi}(x,t;\lambda):=\frac{\sin(\pi\,\alpha)}{\pi}\lambda^{\alpha-1}\int_{0}^te^{-(t-\tau)\lambda}\,\frac{\partial P(x,\tau)}{\partial \tau}\,d\tau =\mu_\alpha(\lambda)\int_{0}^te^{-(t-\tau)\lambda}\,\frac{\partial P(x,\tau)}{\partial \tau}\,d\tau,
\end{equation}
which can be easily shown to satisfy the following PDE, which is locally an ODE:
\begin{eqnarray}
\label{eq:drpsider}
\frac{\partial {\psi}}{\partial t}&=&-\lambda {\psi}+\frac{\sin(\pi\alpha)}{\pi}\lambda^{\alpha-1}\frac{\partial P}{\partial t},\quad \lambda>0,\\
\label{psi_IC}
\psi(x,0;\lambda)&=&0,\quad \lambda>0.
\end{eqnarray}
Consequently, the fractional derivative can be expressed as
\begin{equation}
\label{eq:DRDreppsi}
\frac{\partial^{\alpha}P(x,t)}{\partial t^{\alpha}}=\int_{0}^{\infty} {\psi}(x,t;\lambda)\,d\lambda=(h_\alpha\ast \partial_t P)(t).
\end{equation}
Therefore, \eqref{eq:MaxColeColeP} is equivalent to the augmented system of \eqref{eq:drpsider}, \eqref{psi_IC} and the equation below
\begin{equation}
\tau_0^\alpha \,h_\alpha\ast \partial_t P +P={\epsilon_0(\epsilon_s-\epsilon_{\infty})}E+F_3.
\end{equation}

\subsection{Energy analysis of the Cole-Cole model}
{In \cite{LiHuangLinCOLE}, the following energy function was defined for the system (\ref{eq:MaxColeColeH})-(\ref{eq:MaxColeColeP}) with the periodic boundary condition and the initial conditions  (\ref{eq:iniMax}) under the assumption $F_1=F_2=F_3=0$}
\begin{eqnarray}
\label{def:energyLi}
\widetilde{\mathcal{E}}(H(t),E(t),P(t)):=\int_{\Omega} \epsilon_0\left(\epsilon_s-\epsilon_{\infty}\right)
\left( \epsilon_0\epsilon_{\infty}E^2+\mu_0H^2\right)\,dx
+\int_{\Omega}P^2\,dx,
\end{eqnarray}
and was shown to satisfy the relation
\begin{eqnarray}
\label{def:energyLile}
\widetilde{\mathcal{E}}(H(t),E(t),P(t))\le \widetilde{\mathcal{E}}(H(0),E(0),P(0)),\quad \forall\, t\in[0,\,T].
\end{eqnarray}
However, $\widetilde{\mathcal{E}}$ is not necessarily monotonically decreasing in time. Indeed, as can be seen in Figure \ref{enery_compare} (d)-(f), $\widetilde{\mathcal{E}}$ can be increasing; similar phenomena have been observed in \cite{WangHuang}. In contrast to $\widetilde{\mathcal{E}}$, the energy ${\mathcal{E}}$ defined in the following theorem is guaranteed to be monotonically decreasing.

\begin{theorem} (Monotonically decreasing energy of (\ref{eq:MaxColeColeH})-(\ref{eq:MaxColeColeP}))\label{stableMaxwellCole}
Without forcing, the system (\ref{eq:MaxColeColeH})-(\ref{eq:MaxColeColeP}) with the periodic boundary condition and the initial conditions (\ref{eq:iniMax})  is stable. Specifically, under these assumptions, solutions of the Cole-Cole model satisfy the energy condition 
\begin{equation}\label{eq:dtEnergy}
\frac{d}{d t}\mathcal{E}(H,E,P,\psi)=-\frac{\pi}{\sin(\pi\alpha)}\int_{\Omega}\int_0^{\infty} \frac{\tau_0^{\alpha}}{\epsilon_0(\epsilon_s-\epsilon_{\infty})}
\lambda^{2-\alpha}\psi^2\,d\lambda\,dx\le 0,
\end{equation}
where the energy $\mathcal{E}$ is the sum of the classic energy $\mathcal{E}_1$ and the  diffusive energy $\mathcal{E}_2$
\begin{equation}\label{eq:defeaxctenergy}
\mathcal{E}(H,E,P,\psi)=\mathcal{E}_1(H,E,P)+\mathcal{E}_2(\psi),
\end{equation}
with 
\begin{eqnarray}
\label{def:energy1}
\mathcal{E}_1(H,E,P)&:=&
\frac{1}{2}\int_{\Omega}
 \epsilon_0\epsilon_{\infty}E^2+\mu_0H^2
 +\frac{P^2}{\epsilon_0(\epsilon_s-\epsilon_{\infty})}\,dx,\\
\label{def:energy2}
\mathcal{E}_2(\psi)&:=& \frac{\pi}{2\sin(\pi\alpha)}
\int_{\Omega}\int_0^{\infty}\frac{\tau_0^{\alpha}\lambda^{1-\alpha}}{\epsilon_0(\epsilon_s-\epsilon_{\infty})}\psi^2\,d\lambda\,dx.
\end{eqnarray}
\end{theorem}

\begin{proof}
Clearly,  the function $\mathcal{E}$ is positive definite. Multiplying (\ref{eq:MaxColeColeH}) with $H$  and (\ref{eq:MaxColeColeE}) with $E$, we have
\begin{eqnarray*}
\int_{\Omega} \mu_0 \pr_tH\, H\,dx-\int_{\Omega} \partial_x E \, H\,dx=0,\\
\int_{\Omega}\left(\epsilon_0\epsilon_{\infty}\pr_tE\, E
+\pr_tP\, E\right)\,dx-\int_{\Omega}\partial_x H\,E\,dx=0.
\end{eqnarray*}
Adding the above two equations together, followed by integrating  by parts over $\Omega$ and the periodic  boundary condition, we obtain
\begin{equation}\label{eq:addtwoerr}
\int_{\Omega} \left( \epsilon_0\epsilon_{\infty}\pr_tE\, E+\mu_0\pr_tH\, H\right)\,dx
+\int_{\Omega}
\pr_tP\, E\,dx=0.
\end{equation}
In order to estimate the last terms on the left-hand side of (\ref{eq:addtwoerr}), multiplying
 (\ref{eq:MaxColeColeP}) with $\pr_tP$ and integrating on $\Omega$, we see
$$
 \int_{\Omega} E\, \pr_t P\,dx=\int_{\Omega} \frac{\tau_0^{\alpha}}{\epsilon_0(\epsilon_s-\epsilon_{\infty})} \frac{\pr^{\alpha}P}{\pr t^{\alpha}}\,\partial_tP\,dx+ \int_{\Omega} \frac{P}{\epsilon_0(\epsilon_s-\epsilon_{\infty})}\, \pr_tP\,dx.
$$
Injecting the above equation into (\ref{eq:addtwoerr}) and using (\ref{eq:DRDreppsi}) lead to
\begin{equation*}
\begin{aligned}
&\frac{1}{2}\frac{d}{dt}\int_{\Omega} 
\epsilon_0\epsilon_{\infty}
E^2+\mu_0H^2
+\frac{P^2}{\epsilon_0(\epsilon_s-\epsilon_{\infty})}  \,dx\\
=&\,-\int_{\Omega}  \frac{\tau_0^{\alpha}}{\epsilon_0(\epsilon_s-\epsilon_{\infty})} \frac{\pr^{\alpha}P}{\pr t^{\alpha}}\,\partial_tP\,dx
=-\int_{0}^{\infty} \int_{\Omega} \frac{\tau_0^{\alpha}}{\epsilon_0(\epsilon_s-\epsilon_{\infty})}
\psi
\,\partial_tP\,dx\,d\lambda,
\end{aligned}
\end{equation*}
or equivalently by using (\ref{def:energy1})
\begin{equation}\label{enproof}
\frac{d}{dt}\mathcal{E}_1
=-\int_{0}^{\infty} \int_{\Omega} \frac{\tau_0^{\alpha}}{\epsilon_0(\epsilon_s-\epsilon_{\infty})}
\psi
\,\partial_tP\,dx\,d\lambda.
\end{equation}
Multiplying (\ref{eq:drpsider}) with $ \frac{\tau_0^{\alpha}}{\epsilon_0(\epsilon_s-\epsilon_{\infty})}
\psi$ and integrating on $\Omega$ give
$$
\frac{\pi \lambda^{1-\alpha}}{\sin(\pi\alpha)}
\int_{\Omega} \frac{\tau_0^{\alpha}}{\epsilon_0(\epsilon_s-\epsilon_{\infty})}  \psi\,\pr_t\psi\,dx
+
\frac{\pi \lambda^{2-\alpha}}{\sin(\pi\alpha)}
\int_{\Omega} \frac{\tau_0^{\alpha}}{\epsilon_0(\epsilon_s-\epsilon_{\infty})} \psi^2\,dx
=\int_{\Omega} \frac{\tau_0^{\alpha}}{\epsilon_0(\epsilon_s-\epsilon_{\infty})} \psi\,\pr_t P\,dx,
$$
which together with (\ref{enproof}) and the definition of $\mathcal{E}_2$ implies
$$
\frac{d}{dt}\left(\mathcal{E}_1+\mathcal{E}_2\right)
=-\frac{\pi}{\sin(\pi\alpha)}\int_{\Omega}
\int_0^{\infty}\frac{\tau_0^{\alpha} \lambda^{2-\alpha}}{\epsilon_0(\epsilon_s-\epsilon_{\infty})}
\psi^2\,d\lambda\,dx\le0.
$$
\end{proof}

\begin{rem} The function $\mathcal{E}_1$, which is not monotonically decreasing, is related to the classical energy defined in \cite{LiHuangLinCOLE}. In comparison, the energy derived above is monotonically decreasing because it identifies and includes the diffusive energy $\mathcal{E}_2$. To our knowledge, this is the first time the complete total energy for the Cole-Cole model  is derived. 
The properties of $\mathcal{E}_2 \ge 0$  and $-d\mathcal{E}/dt \ge 0$  motivate us to design a positive-preserving numerical method for the diffusive approximation, so that the diminishing  property of the total energy can be preserved  for the Cole-Cole model.
\end{rem}

\section{Semi-discrete DG scheme}
 \label{SpatialDG}
In this section, we introduce a DG method for spatial discretization of the Cole-Cole model  (\ref{eq:MaxColeColeH})-(\ref{eq:MaxColeColeP}) and establish the optimal error estimate of the semi-discrete scheme.

Suppose the computational domain $\Omega=[x_-,~x_+]$ is partitioned into $x_-=x_{1/2}<x_{2/3}<\cdots<x_{M+1/2}=x_+$ with equal mesh size  $h=(x_+-x_-)/M$. The $j$-th cell is denoted by $I_j=[x_{j-1/2},~x_{j+1/2}]$  with the center $x_j=(x_{j-1/2}+x_{j+1/2})/2$, $j=1,\cdots,M$. The piecewise-polynomial space $V_h$ is defined as
$$
V_h=\left\{v:v|_{I_j}\in \mathbb{P}^k(I_j),\quad j=1,\cdots,M\right\},\  k\in\mathbb{N},\,\  k\ge 1,
$$
where $ \mathbb{P}^k$ is the space of polynomials of degree up to $k$ in each cell $I_j$. For any $v\in V_h$, we denote the limit values of $v$ at $x_{j+1/2}$ from the element $I_{j+1}$ and the element $I_j$ by $v_{j+1/2}^+$ and $v_{j+1/2}^-$, respectively. The jump and average at $x_{j+1/2}$ are denoted by  $[v]_{j+1/2}=v_{j+1/2}^+-v_{j+1/2}^-$ and $\{v\}_{j+\frac{1}{2}}=\frac{v_{j+1/2}^++v_{j+1/2}^-}{2}$, respectively. For the average and jump functions, the following identity holds
\beq\label{eq:fluxsum}
\{a\}[b]+\{b\}[a]=[ab].
\eeq

The DG scheme  for problem  (\ref{eq:MaxColeColeH})-(\ref{eq:MaxColeColeP}) is to find
$U_h(\cdot,t)\in V_h$ with $U=H,\,E,\,P,\,\Psi$ such that for all $v\in V_h$ and all $j$, the following holds
\begin{eqnarray}
\label{eq:semiDGH}
\int_{I_j}\mu_0\frac{\pr  H_h}{\pr t}\,v\,dx
+\int_{I_j} E_h\, \frac{\pr v}{\pr x} \,dx-\widehat{E_h}\,v^{-}|_{x_{j+1/2}}
+\widehat{E_h}\,v^{+}|_{x_{j-1/2}}=\int_{I_j}F_{1}\,v\,dx,\\
\label{eq:semiDGE}
\int_{I_j}\epsilon_0\epsilon_{\infty}\frac{\pr E_h}{\pr t}\,v\,dx +\int_{I_j}\frac{ \pr  P_h}{\pr t}\,v\,dx+\int_{I_j} H_h\, \frac{\pr v}{\pr x}\,dx-\widehat{H_h}v^{-}|_{x_{j+1/2}}
+\widehat{H_h}\,v^{+}|_{x_{j-1/2}}=\int_{I_j}F_{2}\,v\,dx,\\
\label{eq:semiDGP}
\int_{I_j}\int_{0}^{\infty}\frac{\tau_0^{\alpha}}{\epsilon_0(\epsilon_{s}-\epsilon_{\infty})}\psi_h\,v\,d\lambda\,dx+\int_{I_j}\frac{P_h}{\epsilon_0(\epsilon_{s}-\epsilon_{\infty})}\,v\,dx
-\int_{I_j}E_h\,v\,dx=\int_{I_j}\frac{F_{3}}{\epsilon_0(\epsilon_{s}-\epsilon_{\infty})}\,v\,dx,\\
\label{eq:drpsiderdg}
\int_{I_j}\frac{\tau_0^{\alpha}}{\epsilon_0(\epsilon_{s}-\epsilon_{\infty})}\frac{\partial {\psi}_h}{\partial t}\,v\,dx+\lambda \int_{I_j}\frac{\tau_0^{\alpha}}{\epsilon_0(\epsilon_{s}-\epsilon_{\infty})} {\psi}_h\,v\,dx-\frac{\sin(\pi\alpha)}{\pi}\lambda^{\alpha-1}\int_{I_j}\frac{\tau_0^{\alpha}}{\epsilon_0(\epsilon_{s}-\epsilon_{\infty})}\frac{\partial P_h}{\partial t}\,v\,dx=0,
\end{eqnarray}
where  $\hat{\cdot}$ indicates the numerical flux, which is chosen to be the upwind flux 
\beq\label{eq:upwind}
\widehat{E_h}=\{E_h\}
+\frac{1}{2}\sqrt{\frac{\mu_0}{\epsilon_0\epsilon_{\infty}}}\,[H_h],\quad \widehat{H_h}=\{H_h\}
+\frac{1}{2}\sqrt{\frac{\epsilon_0\epsilon_{\infty}}{\mu_0}}\,[E_h].
\eeq
In the above, we have used (\ref{eq:drpsider}) and (\ref{eq:DRDreppsi}). Note that there is a factor of $\frac{\tau_0^{\alpha}}{\epsilon_0(\epsilon_{s}-\epsilon_{\infty})}$ difference between (\ref{eq:drpsider}) and (\ref{eq:drpsiderdg}).

The following theorem shows the energy decreasing property of the discretized system. 

\begin{theorem} (semi-discrete stability)
\label{thm:StaDG}
Under the assumptions $F_1=F_2=F_3=0$, the periodic boundary condition, the initial condition (\ref{eq:iniMax}) and the upwind numerical flux defined in \eqref{eq:upwind}, the solution of the semi-discrete DG scheme (\ref{eq:semiDGH})-(\ref{eq:drpsiderdg}) satisfies the following stability 
\begin{equation*}\label{eq:stasemiDG}
\begin{aligned}
\frac{d }{d t}\mathcal{E}(H_h(t),E_h(t),P_h(t),\psi_h(t))=&-\frac{\pi }{\sin(\pi\alpha)}
\int_{\Omega}\int_{0}^{\infty} \frac{\tau_0^{\alpha}\lambda^{2-\alpha}}{\epsilon_0\left(\epsilon_s-\epsilon_{\infty}\right)}
\psi_h^2(t)\,d\lambda\,dx\\
&-
\sum_{j=1}^M\mathcal{M}_{j+1/2}(H_h(t),\,E_h(t))\le 0,
\end{aligned}
\end{equation*}
where $\mathcal{E}$ is given in (\ref{eq:defeaxctenergy}) and
\begin{equation}
\label{eq:defMmP}
\mathcal{M}_{j+1/2}(H_h,\,E_h)=\frac{1}{2}\left(\sqrt{\frac{\mu_0}{\epsilon_0\epsilon_{\infty}}}
\left[H_h\right]^2
+\sqrt{\frac{\epsilon_0\epsilon_{\infty}}{\mu_0}}
\left[E_h\right]^2\right)_{j+1/2}\ge 0.
\end{equation}
\end{theorem}

\begin{proof}
Setting  $v=H_h$ in (\ref{eq:semiDGH}), $v=E_h$ in (\ref{eq:semiDGE}) and summing up over all elements lead to
\begin{equation}\label{eq:sumelestaHE}
\begin{aligned}
-\int_{\Omega}\pr_tP_hE_h\,dx=
&\int_{\Omega}\mu_0\pr_tH_hH_h\,dx
+\int_{\Omega}\epsilon_0\epsilon_{\infty}\pr_tE_hE_h\,dx \\
&+\sum_{j=1}^M\int_{I_j}\pr_x\left(H_hE_h\right)dx
+\sum_{j=1}^{M}\left(\widehat{E_h}[H_h]
+\widehat{H_h}[E_h]\right)_{j+1/2}.
\end{aligned}
\end{equation}
According to (\ref{eq:fluxsum}) and the numerical flux defined in \eqref{eq:upwind}, we have 
\begin{eqnarray}
\label{eq:sumfluxdg}
\left(\widehat{E_h}[H_h]
+\widehat{H_h}[E_h]\right)_{j+1/2}
=\left[H_hE_h\right]_{j+1/2}+\mathcal{M}_{j+1/2}(H_h,\,E_h),
\end{eqnarray}
where $\mathcal{M}_{j+1/2}(H_h,\,E_h)$ is given in (\ref{eq:defMmP}).
Since
$$
\sum_{j=1}^M\int_{I_j}\pr_x\left(H_hE_h\right)dx
+\sum_{j=1}^{M}[H_hE_h]_{j+1/2}=0,
$$
\eqref{eq:sumelestaHE} becomes
\begin{equation}\label{eq:sumelestaHEred0}
\begin{aligned}
&\int_{\Omega}\mu_0\pr_tH_hH_h\,dx
+\int_{\Omega}\epsilon_0\epsilon_{\infty}\pr_tE_hE_h\,dx
=-\int_{\Omega}\pr_tP_hE_h\,dx-\sum_{j=1}^M\mathcal{M}_{j+1/2}(H_h,\,E_h).
\end{aligned}
\end{equation}
Setting $v=\pr_tP_h$ in (\ref{eq:semiDGP}) and summing over all elements and applying \eqref{eq:sumelestaHEred0} and the definition of $\mathcal{E}_1$ in (\ref{def:energy1}), we obtain
\begin{equation}\label{eq:sumHEcon1e}
\begin{aligned}
\frac{d \mathcal{E}_1(H_h,E_h,P_h)}{d t}
=-\int_{\Omega}\int_{0}^{\infty}\frac{\tau_0^{\alpha}}{\epsilon_0(\epsilon_{s}-\epsilon_{\infty})}\psi_h\,
\pr_tP_h\,d\lambda\,dx
-\sum_{j=1}^M\mathcal{M}_{j+1/2}(H_h,\,E_h).
\end{aligned}
\end{equation}
Setting $v=\psi_h$ in (\ref{eq:drpsiderdg}) and summing up over all elements, we have
\begin{equation*}\label{eq:intpsipp}
\int_{\Omega}\frac{\tau_0^{\alpha}}{\epsilon_0(\epsilon_{s}-\epsilon_{\infty})}\partial_t P_h\,\psi_h\,dx
=\frac{\pi\lambda^{1-\alpha}}{\sin(\pi\alpha)}\int_{\Omega} \frac{\tau_0^{\alpha}}{\epsilon_0(\epsilon_{s}-\epsilon_{\infty})} \partial_t\psi_h\,\psi_h\,dx
+\frac{\pi\lambda^{2-\alpha}}{\sin(\pi\alpha)}\int_{\Omega}\frac{\tau_0^{\alpha}}{\epsilon_0(\epsilon_{s}-\epsilon_{\infty})}\psi_h^2\,dx,
\end{equation*}
and hence \eqref{eq:sumHEcon1e} becomes
\begin{equation}\label{eq:dte1psiptp}
\begin{aligned}
\frac{d \mathcal{E}_1(H_h,E_h,P_h)}{d t}
=&- \frac{\pi}{\sin(\pi\alpha)}
\int_{\Omega}\frac{\tau_0^{\alpha}}{\epsilon_0(\epsilon_{s}-\epsilon_{\infty})}\left(\int_{0}^{\infty}
\lambda^{1-\alpha}\,\partial_t\psi_h\,\psi_h
+\lambda^{2-\alpha}\,\psi_h^2\,d\lambda\right)\,dx\\
&-\sum_{j=1}^M\mathcal{M}_{j+1/2}(H_h,\,E_h).
\end{aligned}
\end{equation}
We complete the proof by combining  (\ref{eq:dte1psiptp}), (\ref{eq:defeaxctenergy})
and  (\ref{def:energy2}).
\end{proof}

In the following, we estimate the errors of the semi-discrete DG scheme. First, we  introduce the standard $L^2$-projection of a function $\varphi\in L^2(\Omega)$ into $V_h$, namely for all $j$
\begin{equation}\label{eq:ProL2}
\int_{I_j}(\mathcal{P}\varphi-\varphi)\,v\,dx=0,\quad \forall \,v\in  \mathbb{P}^{k}(I_j).
\end{equation}
The Gauss-Radau projections $\mathcal{P}^{\pm}$ defined from $\mathcal{H}^1(\Omega)$ onto $V_h$ are defined as
\begin{eqnarray}\label{eq:ProPp}
\int_{I_j}(\mathcal{P}^+\varphi-\varphi)\,v\,dx=0,\quad \mathcal{P}^{+}\varphi(x^{+}_{j-1/2})=\varphi(x_{j-1/2}),\quad  \forall \,v\in \mathbb{P}^{k-1}(I_j),\\
\label{eq:ProPm}
\int_{I_j}(\mathcal{P}^-\varphi-\varphi)\,v\,dx=0,\quad \mathcal{P}^{-}\varphi(x^-_{j+1/2})=\varphi(x_{j+1/2}),\quad\forall \,v\in  \mathbb{P}^{k-1}(I_j).
\end{eqnarray}
For $\mathcal{A}=\mathcal{P}$ or $\mathcal{A}=\mathcal{P}^{\pm}$, it is known that  (see \cite{Ciarlet})
\begin{equation}\label{eq:inequalityP}
\left\|\mathcal{A}\varphi-\varphi\right\|_0
+h\left\|\mathcal{A}\varphi-\varphi\right\|_{L^{\infty}(\Omega)}
+h^{\frac{1}{2}}\left\|\mathcal{A}\varphi-\varphi\right\|_{\Gamma_h}\le Ch^{k+1}\left\|\varphi\right\|_{\mathcal{H}^{k+1}(\Omega)}, \quad \forall\, \varphi\in \mathcal{H}^{k+1}(\Omega),
\end{equation}
and
\begin{equation}\label{eq:StaProine}
\left\|\mathcal{A}\varphi\right\|_{L^{\infty}(\Omega)}\le
\begin{cases}
C\|\varphi\|_{L^{\infty}(\Omega)},   & \forall \, \varphi\in L^{\infty}(\Omega),\quad\quad \mathcal{A}=\mathcal{P},\\
C\|\varphi\|_{L^{\infty}(\Omega)},   & \forall \, \varphi\in W^{1,\infty}(\Omega), \quad \mathcal{A}=\mathcal{P}^{\pm},
\end{cases}
\end{equation}
where  $\Gamma_h$ indicates the set of boundary points of all cell elements and  $C=C(\Omega,k)$ is a positive constant independent of $h$. Here and hereafter, we use $\|\cdot\|$ stands for the $L^2$-norm on $\Omega$. In the implementation,  $\mathcal{P}_P$ and $\mathcal{P}_{\psi}$ are taken to be the $L^2$ projection $\mathcal{P}$, and
\begin{eqnarray}
\label{eq:proPE}
\mathcal{P}_E&=&\frac{1}{2}\mathcal{P}^{+}
\left(E+\sqrt{\frac{\mu_0}{\epsilon_0\epsilon_{\infty}}}\,H\right)+\frac{1}{2}\mathcal{P}^{-}
\left(E-\sqrt{\frac{\mu_0}{\epsilon_0\epsilon_{\infty}}}\,H\right),\\
\label{eq:proPH}
\mathcal{P}_H&=&\frac{1}{2}\mathcal{P}^{+}\left(H+
\sqrt{\frac{\epsilon_0\epsilon_{\infty}}{\mu_0}}\,E\right)
+\frac{1}{2}\mathcal{P}^{-}\left(H-\sqrt{\frac{\epsilon_0\epsilon_{\infty}}{\mu_0}}\,E\right).
\end{eqnarray}

\begin{theorem} (semi-discrete error estimate)
\label{thm:ErrorBDF2DGAPP}
Let $H,\, E,\, P$ be the solutions  of the Cole-Cole model  (\ref{eq:MaxColeColeH})-(\ref{eq:MaxColeColeP}) and $H_h,\,E_h,\,P_h$ be the solutions of  the semi-discrete DG scheme (\ref{eq:semiDGH})-(\ref{eq:drpsiderdg}) with constant coefficients and the initial data given in (\ref{eq:iniMax}). If $\psi_h(0)=0$ and  the initial data of $P_h(0)$, $E_h(0)$  and $H_h(0)$ are obtained by the projections of  $\mathcal{P}_P$, $\mathcal{P}_E$ and $\mathcal{P}_H$, respectively, then the semi-discrete DG scheme (\ref{eq:semiDGH})-(\ref{eq:drpsiderdg}) with the periodic boundary condition and the upwind numerical flux \eqref{eq:upwind} satisfies the following error estimate
\beq\label{eq:ErrorBDF2DGAPP}
\begin{aligned}
\left\|E(t)-E_h(t)\right\|^2
+\left\|H(t)-H_h(t)\right\|^2
+\left\|P(t)-P_h(t)\right\|^2\le C h^{2k+2},\quad \forall\, t\ge0,
\end{aligned}
\eeq
where $C$ is a positive constant independent of $h$.
\end{theorem}

\begin{proof}
With the numerical flux being consistent, the semi-discrete  numerical scheme (\ref{eq:semiDGH})-(\ref{eq:drpsiderdg})   also holds when we replace the numerical solutions  with the exact solutions. That is to say for all $j$, it holds that
\begin{eqnarray}
\label{eq:semiDGHconerr}
\int_{I_j}\mu_0\frac{\pr  H}{\pr t}\,v\,dx
+\int_{I_j} E\, \frac{\pr v}{\pr x} \,dx-\widehat{E}\,v^{-}|_{x_{j+1/2}}
+\widehat{E}\,v^{+}|_{x_{j-1/2}}=\int_{I_j}F_1\,v\,dx,\\
\label{eq:semiDGEconerr}
\int_{I_j}\epsilon_0\epsilon_{\infty}\frac{\pr E}{\pr t}\,v\,dx +\int_{I_j}\frac{ \pr  P}{\pr t}\,v\,dx+\int_{I_j} H\, \frac{\pr v}{\pr x}\,dx-\widehat{H}v^{-}|_{x_{j+1/2}}
+\widehat{H}\,v^{+}|_{x_{j-1/2}}=\int_{I_j}F_2\,v\,dx,\\
\label{eq:semiDGPconerr}
\int_{I_j}\int_{0}^{\infty}\frac{\tau_0^{\alpha}}{\epsilon_0(\epsilon_{s}-\epsilon_{\infty})} {\psi}\,v\,d\lambda\,dx+\int_{I_j}\frac{P}{\epsilon_0(\epsilon_{s}-\epsilon_{\infty})}\,v\,dx
=\int_{I_j}E\,v\,dx+\int_{I_j}\frac{F_3}{\epsilon_0(\epsilon_{s}-\epsilon_{\infty})}\,v\,dx,\\
\label{eq:drpsiderdgconerr}
\int_{I_j}\frac{\tau_0^{\alpha}}{\epsilon_0(\epsilon_{s}-\epsilon_{\infty})}\frac{\partial {\psi}}{\partial t}\,v\,dx+\lambda \int_{I_j}\frac{\tau_0^{\alpha}}{\epsilon_0(\epsilon_{s}-\epsilon_{\infty})} {\psi}\,v\,dx=\frac{\sin(\pi\alpha)}{\pi}\lambda^{\alpha-1}\int_{I_j}\frac{\tau_0^{\alpha}}{\epsilon_0(\epsilon_{s}-\epsilon_{\infty})}\frac{\partial P}{\partial t}\,v\,dx,
\end{eqnarray}
where $E\in C^1(0,T;\mathcal{H}^1(\Omega)),\,H\in C^1(0,T;\mathcal{H}^1(\Omega)),\,P\in H^1(0,T; L^2(\Omega) )$.

We decompose the error $U-U_h$, $U=H,\,E,\,P,\,\Psi$ as 
\begin{equation}\label{def:err}
\eta_U:=\mathcal{P}_UU-U_h\,\in V_h,\quad \xi_U:=\mathcal{P}_UU-U,\quad U-U_h=\eta_U-\xi_U.
\end{equation}
Note that the choice of projections \eqref{eq:proPE} and \eqref{eq:proPH} imply $\widehat{\xi}_E=\widehat{\xi}_H=0$. Furthermore, the linearity of the projections $\mathcal{P}$ and $\mathcal{P}^{\pm}$ leads to $\partial_t\xi_U=\xi_{\partial_t U}$, $U=H,\,E,\,P,\, \Psi$. 

Subtracting (\ref{eq:semiDGH})-(\ref{eq:drpsiderdg}) from (\ref{eq:semiDGHconerr})-(\ref{eq:drpsiderdgconerr}) gives for all $v\in V_h$
\begin{subequations}
\begin{align}
\label{eq:BDF2DGHerr}
&\,\nonumber\int_{I_j}\mu_0\partial_t\eta_H\,v\,dx
+ \int_{I_j}  \eta_E\, \pr_xv \,dx-\widehat{\eta_E}\,v^-|_{x_{j+1/2}}
+\widehat{\eta_E}\,v^+|_{x_{j-1/2}}\\
=&\,\int_{I_j}\mu_0\pr_t \xi_H\,v\,dx
+ \int_{I_j}\xi_E\,\pr_x v\,dx
-\widehat{\xi_E}\,
v^-|_{x_{j+1/2}}+\widehat{\xi_E}\,
v^+|_{x_{j-1/2}},\\
\label{eq:BDF2DGEerr}
&\nonumber\int_{I_j}\epsilon_0\epsilon_{\infty}\partial_t\eta_E\,v\,dx+\int_{I_j}\partial_t \eta_P\,v\,dx
+ \int_{I_j}  \eta_H\, \pr_xv \,dx-\widehat{\eta_H}\,v^{-}|_{x_{j+1/2}}
+\widehat{\eta_H}\,v^{+}|_{x_{j-1/2}}\\
=&\,\int_{I_j}\epsilon_0\epsilon_{\infty}\pr_t\xi_E\,v\,dx
+\int_{I_j}\pr_t \xi_P\,v\,dx
+\int_{I_j}\xi_H\,\pr_x v\,dx
- \widehat{\xi_H}
\,v^{-}|_{x_{j+1/2}}+\widehat{\xi_H}
\,v^{+}|_{x_{j-1/2}},\\
\label{eq:BDF2DGPerr}
&\nonumber \int_{I_j}\int_0^{\infty}\frac{\tau_0^{\alpha}}{\epsilon_0(\epsilon_{s}-\epsilon_{\infty})}\eta_{\psi}\,v\,d\lambda\,dx
+ \int_{I_j}\frac{\eta_P}{\epsilon_0(\epsilon_{s}-\epsilon_{\infty})}\, v\,dx-\int_{I_j}
\eta_E\,v\,dx\\
=&\,
\int_{I_j}\int_0^{\infty}\frac{\tau_0^{\alpha}}{\epsilon_0(\epsilon_{s}-\epsilon_{\infty})}\xi_{\psi}\,v\,d\lambda\,dx
+ \int_{I_j}\frac{\xi_P}{\epsilon_0(\epsilon_{s}-\epsilon_{\infty})}\, v\,dx
-\int_{I_j}
\xi_E\,v\,dx,\\
\label{eq:psierrpr}
&\nonumber \int_{I_j}\frac{\tau_0^{\alpha}}{\epsilon_0(\epsilon_{s}-\epsilon_{\infty})}\partial_t \eta_{\psi}\,v\,dx+\lambda\int_{I_j}\frac{\tau_0^{\alpha}}{\epsilon_0(\epsilon_{s}-\epsilon_{\infty})}\eta_{\psi}\,v\,dx
-\frac{\sin(\pi\alpha)}{\pi}\lambda^{\alpha-1}\int_{I_j}\frac{\tau_0^{\alpha}}{\epsilon_0(\epsilon_{s}-\epsilon_{\infty})}\partial_t\eta_P\,v\,dx\\
=&\,\int_{I_j}\frac{\tau_0^{\alpha}}{\epsilon_0(\epsilon_{s}-\epsilon_{\infty})}\partial_t \xi_{\psi}\,v\,dx+\lambda\int_{I_j}\frac{\tau_0^{\alpha}}{\epsilon_0(\epsilon_{s}-\epsilon_{\infty})}\xi_{\psi}\,v\,dx
-\frac{\sin(\pi\alpha)}{\pi}\lambda^{\alpha-1}\int_{I_j}\frac{\tau_0^{\alpha}}{\epsilon_0(\epsilon_{s}-\epsilon_{\infty})}\partial_t\xi_P\,v\,dx.
\end{align}
\end{subequations}
Following almost the same line as the proof of Theorem \ref{thm:StaDG}, we can derive
\begin{equation}\label{eq:semiDGerrte}
\begin{aligned}
\frac{d \mathcal{E}(\eta_H,\eta_E,\eta_P,\,\eta_\Psi)}{d t}=&
-\frac{\pi }{\sin(\pi\alpha)}
\int_{\Omega}\int_{0}^{\infty}\frac{\tau_0^{\alpha}\lambda^{2-\alpha}}{\epsilon_0(\epsilon_s-\epsilon_{\infty})}
\,\eta_\psi^2\,d\lambda\,dx\\
&-\sum_{j=1}^M\mathcal{M}_{j+1/2}(\eta_H,\,\eta_E)+\sum_{i=1}^4\mathcal{J}_i
\le\sum_{i=1}^4\mathcal{J}_i,
\end{aligned}
\end{equation}
where  $\mathcal{M}_{j+1/2}(\eta_H,\,\eta_E)\ge0$ is given in (\ref{eq:defMmP}), $\mathcal{E}$ given in  (\ref{eq:defeaxctenergy}) and
\begin{eqnarray*}
\mathcal{J}_{1}&=&\int_{\Omega}\left(\mu_0\partial_t \xi_H\,\eta_H+\epsilon_0\epsilon_{\infty}\partial_t \xi_E\,\eta_E+\partial_t\xi_P\eta_E\,\right)\,dx,\\
\mathcal{J}_{2}&=& \sum_{j=1}^M\int_{I_j}\left(\xi_E\partial_x\eta_H+\xi_H\partial_x\eta_E\right)\,dx
+\sum_{j=1}^M\left(\widehat{\xi_E}[\eta_H]+\widehat{\xi_H}[\eta_E]\right)_{j+1/2},\\
\mathcal{J}_{3}&=&\sum_{j=1}^M
\int_{I_j}\frac{1}{\epsilon_0(\epsilon_s-\epsilon_{\infty})}\left[\xi_P\partial_t\eta_P
+\tau_0^{\alpha}\int_0^{\infty}\xi_{\psi}\partial_t\eta_P\,d\lambda\right]\,dx-\int_{I_j}\xi_E\partial_t\eta_P\,dx,\\
\mathcal{J}_{4}&=&\sum_{j=1}^M
\int_{0}^{\infty}\int_{I_j}\frac{\tau_0^{\alpha}}{\epsilon_0(\epsilon_s-\epsilon_{\infty})}\left[\frac{\pi}{\sin(\pi\alpha)}\left(\lambda^{2-\alpha}
\xi_{\psi}\eta_{\psi}+ \lambda^{1-\alpha}\partial_t\xi_{\psi}\eta_{\psi}\right)
-\partial_t\xi_P\eta_{\psi}\right]\,dx\,d\lambda.
\end{eqnarray*}

In the following, we estimate $\mathcal{J}_i$ with $i=1,\cdots,4$.
According to the projection properties (\ref{eq:ProPp})-(\ref{eq:ProPm}) of $\mathcal{P}^{\pm}$, the Cauchy-Schwartz inequality, we have
$$
\mathcal{J}_1\le  Ch^{2k+2}+\frac{1}{2}\int_{\Omega}\mu_0\eta^2_H
+\epsilon_0\epsilon_{\infty}\eta^2_E\,dx.
$$
With the choice of $\mathcal{P}_U,\,U=E,H,P,\psi$ and  the projection properties (\ref{eq:ProL2})-(\ref{eq:ProPm})
and the assumption of constant coefficients, we have
\begin{eqnarray*}
\begin{aligned}
\mathcal{J}_2&=\mathcal{J}_4=0.
\end{aligned}
\end{eqnarray*}
Integrating $\mathcal{J}_3$ in $[0,\,t]$ and taking into account the projection property (\ref{eq:ProL2}) and $\eta_P(x,0)=0$, followed by an application of Young's inequality lead to
\begin{eqnarray*}\label{eq:J3intA}
\begin{aligned}
\int_0^t\mathcal{J}_3\,d\tau &= -\int_0^t\int_{\Omega}\xi_E\partial_t\eta_P\,dx d\tau
= \int_0^t\int_{\Omega}\partial_t\xi_E\eta_P\,dx\,d\tau-\int_{\Omega}\xi_E\eta_Pdx\\
&\le C(\Omega,k,t)h^{2k+2}+\frac{1}{4}\int_{\Omega}\frac{\eta^2_P}{\epsilon_0(\epsilon_s-\epsilon_{\infty})}\,dx
+\frac{1}{4}\int_0^t\int_{\Omega}\frac{\eta_P^2}{\epsilon_0(\epsilon_s-\epsilon_{\infty})}\,dx\,d\tau.
\end{aligned}
\end{eqnarray*}
By combining the above estimates with (\ref{eq:semiDGerrte}), it can be concluded that
\begin{eqnarray*}
\overline{\mathcal{E}}(\eta_H(t),\eta_E(t),\eta_P(t))&\le&\mathcal{E}(\eta_H(0),\eta_E(0),\eta_P(0),\, \eta_\Psi(0))\\
&+&C(\Omega,k,t)h^{2k+2}+\int_0^t \overline{\mathcal{E}}(\eta_H(\tau),\eta_E(\tau),\eta_P(\tau))\,d\tau,
\end{eqnarray*}
where we have set 
$$
\overline{\mathcal{E}}(\eta_H(t),\eta_E(t),\eta_P(t)):=
\frac{1}{2}\int_{\Omega}\mu_0\eta^2_H
+\epsilon_0\epsilon_{\infty}\eta^2_E\,dx+\frac{1}{4}\int_{\Omega}\frac{\eta^2_P}{\epsilon_0(\epsilon_s-\epsilon_{\infty})}\,dx\ge 0.
$$
By our assumptions and the projection estimate (\ref{eq:inequalityP}), it is easy to check
$$
\mathcal{E}(\eta_H(0),\eta_E(0),\eta_P(0),\,\eta_\Psi(0))\le Ch^{2k+2},
$$
which further indicates 
\begin{equation*}
\overline{\mathcal{E}}(\eta_H(t),\eta_E(t),\eta_P(t))\le C(\Omega,k,t)h^{2k+2}+\int_0^t \overline{\mathcal{E}}(\eta_H(t),\eta_E(t),\eta_P(t))\,d\tau.
\end{equation*}
Applying the Gronwall inequality, we obtain
\begin{equation}\label{eq:resetaerr}
\overline{\mathcal{E}}(\eta_H(t),\eta_E(t),\eta_P(t))\le C(\Omega,k,t)h^{2k+2}.
\end{equation}
Recalling that $U(t)-U_h(t)=\eta_U-\xi_U$ with $U=H,\,E,\,P$, we complete the proof by combining the triangle inequality, the projection estimate (\ref{eq:inequalityP}) and estimate (\ref{eq:resetaerr}).
\end{proof}

\section{The diffusive approximation equations}
\label{appmodel}
In this section, we derive the approximate equations of the Cole-Cole model (\ref{eq:MaxColeColeH})-(\ref{eq:MaxColeColeP}).
Based on the quadrature formula for \eqref{eq:DRDreppsi} with $\zeta_\ell,\,\lambda_\ell$ being the weights and abscissae, respectively, the  diffusive representation of the fractional derivative (\ref{eq:DRDreppsi}) is approximated by
\begin{eqnarray}
\label{eq:appderfrsum}
\frac{\partial^{\alpha} P(x,t)}{\partial t^{\alpha}}&=&\int_{0}^{\infty}\psi(x,t;\lambda)\,d\lambda
\approx \sum_{\ell=1}^L\zeta_{\ell}\psi(x,t;\lambda_\ell)
=:\sum_{\ell=1}^L\zeta_{\ell}\psi_\ell, \\
\label{eq:sumpointode}
\frac{\partial \psi_\ell}{\partial t}&=&-\lambda_\ell\psi_\ell+\frac{\sin(\pi\alpha)}{\pi}\lambda_\ell^{\alpha-1}\frac{\partial P}{\partial t}, \\
\psi_\ell(x,0;\lambda_\ell)&=&0,
\end{eqnarray}
where the auxiliary  variables  $\psi_\ell$ is defined in (\ref{def:psidr}). By this approach, the fractional derivative involved in the polarization equation  is replaced by a finite number of auxiliary  variables, which satisfy the evolution equation (\ref{eq:sumpointode}).
Injecting (\ref{eq:appderfrsum}) into (\ref{eq:MaxColeColeP}), we have the following approximate polarization equation
\begin{equation}\label{eq:apppolarder}
\tau_0^{\alpha}\sum_{\ell=1}^L\zeta_\ell\psi_\ell+P=\epsilon_0(\epsilon_s-\epsilon_{\infty})E+F_3.
\end{equation}
Combining (\ref{eq:MaxColeColeH})-(\ref{eq:MaxColeColeE}), (\ref{eq:appderfrsum})-(\ref{eq:sumpointode}) and (\ref{eq:apppolarder}), we obtain the following approximate system
\begin{eqnarray}
\label{eq:MaxColeColeHapp}
\mu_0 \frac{\pr H}{\pr t}&=& \frac{\partial E}{\partial x}+ {F}_1, \, \quad \quad \quad \quad \quad \quad\quad\, \text{in}\quad \Omega\times (0,T],\\
\label{eq:MaxColeColeEapp}
\epsilon_0\epsilon_{\infty}\frac{\pr E}{\pr t} &=&\frac{ \partial H}{\partial x} -\frac{\pr P}{\pr t}+{F}_2,\,\ \, \quad \quad \quad \quad \text{in}\quad \Omega\times (0,T],\\
\label{eq:MaxColeColePapp}
\tau_0^{\alpha}\sum_{\ell=1}^L\zeta_\ell\psi_\ell+P&=&\epsilon_0(\epsilon_s-\epsilon_{\infty})E+F_3,\ \ \quad\quad\quad \text{in}\quad \Omega\times (0,T],\\
\label{eq:MaxColeColepsidapp}
\frac{\partial \psi_\ell}{\partial t}+\lambda_\ell\psi_\ell&=&\frac{\sin(\pi\alpha)}{\pi}\lambda_\ell^{\alpha-1}\frac{\partial P}{\partial t},\ \ \quad \quad\quad \text{in}\quad \Omega\times (0,T].\label{Psi_ODE}
\end{eqnarray}
The DG scheme  for problem  (\ref{eq:MaxColeColeHapp})-(\ref{eq:MaxColeColepsidapp}) is to find $U_h(\cdot,t)\in V_h$ with $U=H,\,E,\,P,\,\psi_\ell$ such that for all $v\in V_h$ and all $j$, the following holds
\begin{eqnarray*}
\label{eq:semiDGHappd}
\int_{I_j}\mu_0\frac{\pr  H_h}{\pr t}\,v\,dx
+\int_{I_j} E_h\, \frac{\pr v}{\pr x} \,dx-\widehat{E_h}\,v^{-}|_{x_{j+1/2}}
+\widehat{E_h}\,v^{+}|_{x_{j-1/2}}=\int_{I_j}F_{1}\,v\,dx,\\
\label{eq:semiDGEappd}
\int_{I_j}\epsilon_0\epsilon_{\infty}\frac{\pr E_h}{\pr t}\,v\,dx +\int_{I_j}\frac{ \pr  P_h}{\pr t}\,v\,dx+\int_{I_j} H_h\, \frac{\pr v}{\pr x}\,dx-\widehat{H_h}v^{-}|_{x_{j+1/2}}
+\widehat{H_h}\,v^{+}|_{x_{j-1/2}}=\int_{I_j}F_{2}\,v\,dx,\\
\label{eq:semiDGPappd}
\int_{I_j}\frac{\tau_0^{\alpha}}{\epsilon_0(\epsilon_{s}-\epsilon_{\infty})}\sum_{\ell=1}^L\zeta_\ell\psi_{\ell,h}\,v\, dx+\int_{I_j}\frac{P_h}{\epsilon_0(\epsilon_{s}-\epsilon_{\infty})}\,v\,dx
-\int_{I_j}E_h\,v\,dx=\int_{I_j}\frac{F_{3}}{\epsilon_0(\epsilon_{s}-\epsilon_{\infty})}\,v\,dx,\\
\label{eq:drpsiderdgappd}
\int_{I_j}\frac{\tau_0^{\alpha}}{\epsilon_0(\epsilon_{s}-\epsilon_{\infty})}\frac{\partial {\psi}_{\ell,h}}{\partial t}\,v\,dx+\lambda \int_{I_j}\frac{\tau_0^{\alpha}}{\epsilon_0(\epsilon_{s}-\epsilon_{\infty})} {\psi}_{\ell,h}\,v\,dx-\frac{\sin(\pi\alpha)}{\pi}\lambda_\ell^{\alpha-1}\int_{I_j}\frac{\tau_0^{\alpha}}{\epsilon_0(\epsilon_{s}-\epsilon_{\infty})}\frac{\partial P_h}{\partial t}\,v\,dx=0,
\end{eqnarray*}
where the numerical flux $\hat{\cdot}$ is given in (\ref{eq:upwind}).
It is this system  (\ref{eq:MaxColeColeHapp})-\eqref{Psi_ODE} that forms the basis for our numerical work,  in which
the fractional derivative  is approximated by  a series of local ordinary differential equations.   Recalling that
the diffusive approximation is defined on $(0,\,\infty)$, one requirement is that  the abscissae obtained numerically should satisfy  $\lambda_\ell>0$ automatically.

The stability of the exact Cole-Cole model \eqref{eq:MaxColeColeH}-\eqref{eq:MaxColeColeP}is analysed in Theorem \ref{stableMaxwellCole}. We now illustrate  the energy property of the approximate   system (\ref{eq:MaxColeColeHapp})-(\ref{eq:MaxColeColepsidapp}). Similar to  the definition of the total energy $\mathcal{E}$, we define $\mathcal{E}^{\sharp}_1=\mathcal{E}_1$  in (\ref{def:energy1}) and introduce
\begin{eqnarray}\label{eq:energyE2primapp}
\mathcal{E}^{\sharp}_2&=&
\frac{\pi}{2\sin(\pi\alpha)}
\int_{\Omega}\tau_0^{\alpha}\sum_{\ell=1}^L\zeta_\ell \lambda_\ell^{1-\alpha} \left|\psi_\ell\right|^2\,dx,\\
\label{eq:energyappTotalEp}
\mathcal{E}^{\sharp}&=&\mathcal{E}_1^{
\sharp}+\mathcal{E}^{\sharp}_2.
\end{eqnarray}
Suppose $F_1=F_2=F_3=0$ and consider the periodic  boundary condition  and initial conditions (\ref{eq:iniMax}). Following almost the same line as the proof of  Theorem \ref{stableMaxwellCole}, the following relation can be proved.
\begin{equation}\label{eq:disappenge2}
\frac{d}{d t}\mathcal{E}^{\sharp}=
-\frac{\pi }{\sin(\pi\alpha)}\int_{\Omega}\tau_0^{\alpha}
\sum_{\ell=1}^L\zeta_{\ell}\lambda_\ell^{2-\alpha}
\psi_\ell\,\psi_\ell\,dx.
\end{equation}

\begin{theorem} (stability analysis of (\ref{eq:MaxColeColeHapp})-(\ref{eq:MaxColeColepsidapp}))\label{stableMaxwellColeapp}
Under the assumptions of $F_1=F_2=F_3=0$ and  $\zeta_\ell>0,\,\lambda_\ell>0$, with 
the periodic boundary condition  and the initial conditions (\ref{eq:iniMax}), the approximate system  (\ref{eq:MaxColeColeHapp})-(\ref{eq:MaxColeColepsidapp})  is stable, in the sense that the total  energy $\mathcal{E}^{\sharp}$ is positive-definite and decreasing.
\end{theorem}

\begin{proof}
The result follows from (\ref{eq:energyE2primapp})-(\ref{eq:disappenge2}) with $\zeta_\ell>0$ and $\lambda_\ell>0$.
\end{proof}

\begin{rem}
From the above theorem, the conditions $\lambda_\ell>0$ alone do not guarantee that   the approximate  system (\ref{eq:MaxColeColeHapp})-(\ref{eq:MaxColeColepsidapp}) is stable.  However, if we have $\zeta_\ell>0$ and $\lambda_\ell>0$, then the approximate   system is stable with
$\mathcal{E}^{\sharp}$ being  positive-definite and satisfying $d\mathcal{E}^{\sharp}/dt \le 0$.
Therefore, numerical methods with high accuracy and positive preserving of weights $\zeta_\ell$ and abscissae $\lambda_\ell$ is of utmost importance for the diffusive approximation.
\end{rem}

\section{Optimization method for quadrature coefficients}
\label{optimization}

In order to simulate the approximate  system  (\ref{eq:MaxColeColeHapp})-(\ref{eq:MaxColeColepsidapp}), the  weights $\zeta_\ell$ and abscissae $\lambda_\ell$  are need to be  determined.
In this section, we introduce the nonlinear constrained optimization method for  diffusive approximation so that the requirements of  $\zeta_\ell>0$  and $\lambda_\ell>0$ are satisfied.
For the sake of clarity, the space coordinate $x$ is omitted.

Since $P_0=0$, by using Laplace transform, we get
$$
\mathcal{L}\left[\frac{\partial^{\alpha} {P}}{\partial t^{\alpha}}\right]=\left(i\omega\right)^{\alpha} \mathcal{L}\left[ P\right].
$$
By \eqref{def:h_alpha}, \eqref{eq:DRDreppsi} and \eqref{eq:appderfrsum} we have 
$$
\mathcal{L}\left[\frac{\partial^{\alpha} {P}}{\partial t^{\alpha}}\right]\approx
 \frac{\sin(\pi\alpha)}{\pi}(i\omega)
\sum_{\ell=1}^L\frac{\zeta_\ell\lambda_\ell^{\alpha-1}}{i\omega+\lambda_\ell}\, \mathcal{L}\left[ P\right].
$$
Suppose the  frequency range $[\omega_{min},\, \omega_{max}]$ is divided into log-spaced sample points with
$$
\omega_m=\omega_{min}\left(\frac{\omega_{max}}{\omega_{min}}\right)^{\frac{m-1}{M-1}},\quad m=1,\,\cdots,M>1.
$$
Denote
\begin{equation}\label{eq:errfunopt}
\chi(\omega):=\frac{\mathcal{B}(\omega)}{(i\omega)^{\alpha}},\quad
\mathcal{B}(\omega)=\frac{\sin(\pi\alpha)}{\pi}(i\omega)
\sum_{\ell=1}^L\frac{\zeta_\ell\lambda_\ell^{\alpha-1}}{i\omega+\lambda_\ell},
\end{equation}
then $\chi(\omega)$ should be 1. As a result, we  consider the following objective function
\begin{equation}\label{eq:defobjdel}
\delta^2=\sum_{m=1}^M\left|\chi(\omega_m)-1\right|^2
=\sum_{m=1}^M\left|\frac{\sin(\pi\alpha)}{\pi}(i\omega_m)^{1-\alpha}
\sum_{\ell=1}^L\frac{\zeta_\ell\lambda_\ell^{\alpha-1}}{i\omega_m+\lambda_\ell}-1\right|^2.
\end{equation}
From Theorem \ref{stableMaxwellColeapp}, we know that the weights $\zeta_\ell>0$ and abscissae $\lambda_\ell>0$ play a key role for the stability of  the approximate system (\ref{eq:MaxColeColeHapp})-(\ref{eq:MaxColeColepsidapp}).  With this in mind, we can establish the  following constrained optimization problem
\begin{numcases}{}
\label{eq:obj1nonoptp}
\min_{\zeta_\ell,\,\lambda_\ell} \delta^2,\\
\label{eq:nolinoptimlam}
0<\lambda_\ell<10\,\omega_{max },\\
\label{eq:obj1nonoptweightp0}
\zeta_\ell>0,\quad \ell=1,\,\cdots,L,
\end{numcases}
which leads to a square system when $M=L$ and an overdetermined system when $M>L$. We would like to point out that there is no theoretical result for the choice of the interval $[\omega_{min},\,\omega_{max}]$ and the log-spaced sample points.
In the implementation, we use the program SolvOpt \cite{KappelSolvOpt,ShorSolvOpt} to solve the above nonlinear constrained optimization problem. The  modified Gauss-Jacobi quadrature formula is used to initialize the  iterative process \cite{BirkgauJa}, with which the positivity of the  starting points is satisfied. More precisely, by introducing
$
\lambda=\left(\frac{1-\tilde{\lambda}}{1+\tilde{\lambda}}\right)^2,
$
we  derive
\begin{equation}\label{eq:mdfgauuJa}
\begin{aligned}
\int_{0}^{\infty}{\psi}(t;\lambda)\,d\lambda
&=\int_{-1}^1\frac{4(1-\tilde{\lambda})}{(1+\tilde{\lambda})^3}\,
\psi\left(t;\left(\frac{1-\tilde{\lambda}}{1+\tilde{\lambda}}\right)^2\right)\,d\tilde{\lambda}\\
&=\int_{-1}^1(1+\tilde{\lambda})^{\nu_1}\,(1-\tilde{\lambda})^{\nu_2}\,
\tilde{\psi}(t;\tilde{\lambda})\,d\tilde{\lambda}\simeq \sum_{\ell=1}^L \tilde{\zeta}_\ell\,\tilde{\psi}(t;\tilde{\lambda}_\ell).
\end{aligned}
\end{equation}
As shown in \cite{BirkgauJa}, an optimal choice is $\nu_1=1-2\,\bar{\alpha},\,\nu_2=1+2\,\bar{\alpha}$ with $\bar{\alpha}=2\,\alpha-1$, also see \cite{Lombardalpha}. After we obtain the Gauss quadrature, we have
$$
\lambda_\ell=\left(\frac{1-\tilde{\lambda}_\ell}{1+\tilde{\lambda}_\ell}\right)^2,\quad
\zeta_\ell
=\frac{4\tilde{\zeta}_\ell}{(1+\tilde{\lambda}_\ell)^{\nu_1+3}\,(1-\tilde{\lambda}_\ell)^{\nu_2-1}}.
$$

To  assess the overall performances of the proposed nonlinear constrained optimization method, we firstly report the relative errors between $(i\,\omega)^{\alpha}$ and $\mathcal{B}(\omega)$ with fixed $\alpha=0.5$ and frequency bands $[0.5,\, 5]$ and $[20\pi,\,2000\pi]$. Fig.\,\ref{fig:relativeerror} displays  the influence of $M$ and $L$ on the accuracy of the optimization procedure,  which shows that  the proposed approach approximate the infinite integral  very well with both square system ($M=L$) and overdetermined system ($M>L$). The error curves are plotted in $[\omega_{min}/2,\,2\,\omega_{max}]$  and the error curves of  $\alpha\in(0,\,1)$ is very similar to $\alpha=0.5$.
Secondly, we examine the sign of quadrature coefficients.
In Fig.\,\ref{fig:positivesign}, we present the  values of the weights $\zeta_\ell$ and abscissae $\lambda_\ell$  for different $\alpha$ and  $L$. We can observe that  some   quadrature coefficients are very  close to zero but they are always larger than zero, which implies that the proposed nonlinear optimization method  preserves the positivity of $\zeta_\ell$ and $\lambda_\ell$ so that the approximate system (\ref{eq:MaxColeColeHapp})-(\ref{eq:MaxColeColepsidapp}) is stable with  $\mathcal{E}_2^{\sharp}\ge 0$ and $-d\mathcal{E}^{\sharp}/dt\ge 0$.

Since  the relative errors  obtained by the square system  and  the overdetermined systems are comparable, in the forthcoming  numerical experiments, we will always set $M=2L$.
\begin{figure}[htb]
\centering
\subfloat[$M=L$]{%
   \includegraphics[width=2in,height=1.5in]{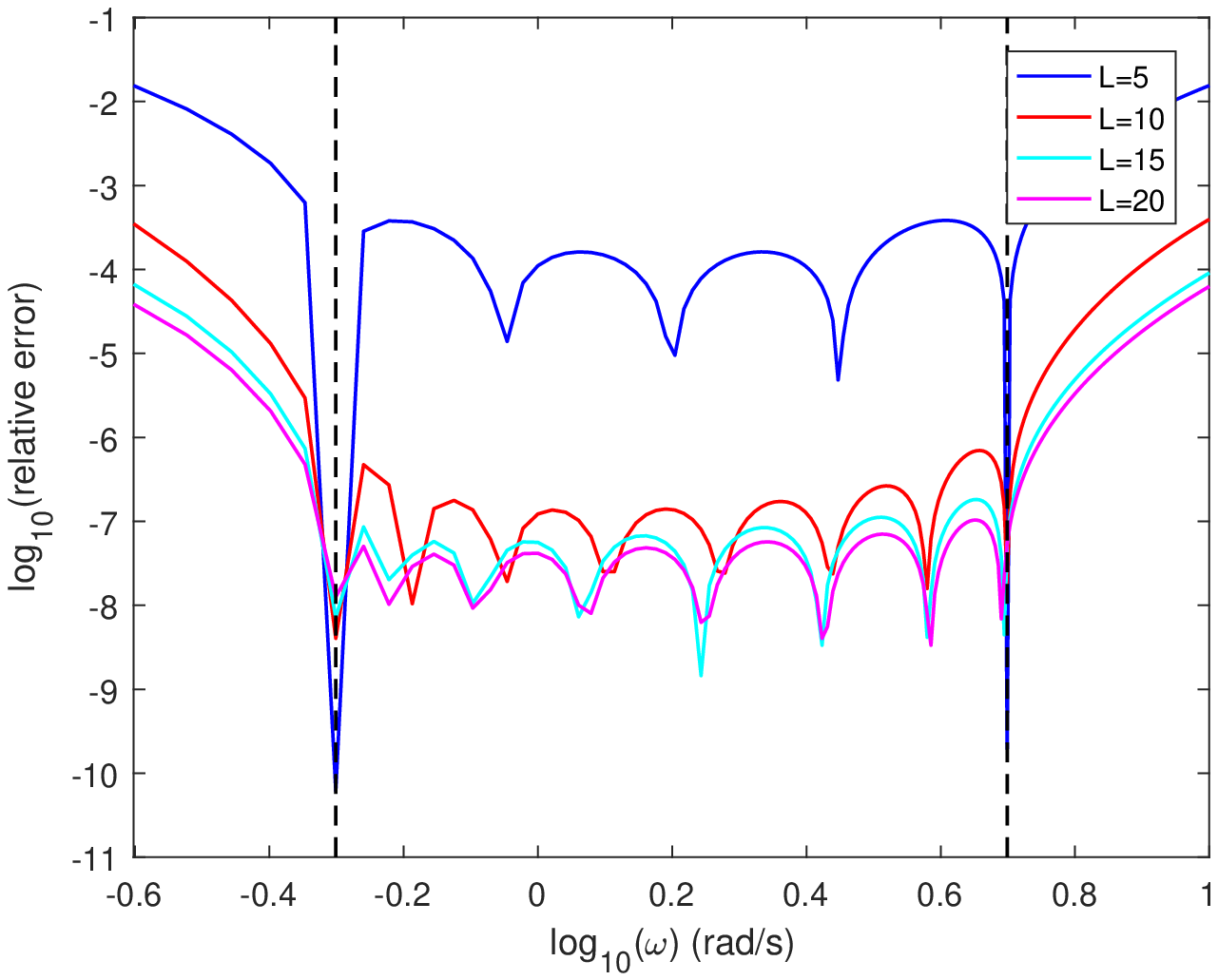}}
\subfloat[$M=2L$]{%
   \includegraphics[width=2in,height=1.5in]{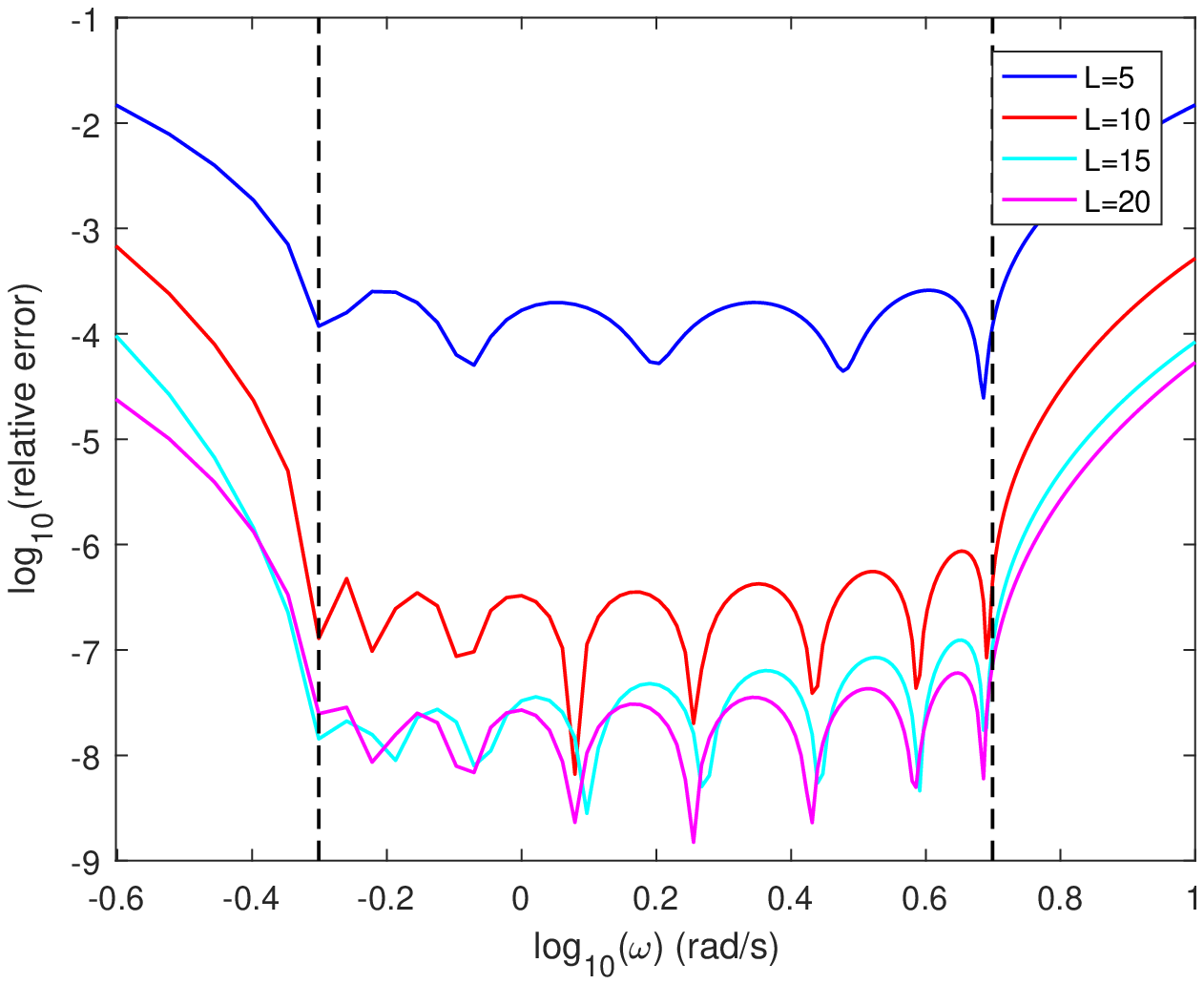}}
\subfloat[$M=3L$]{%
   \includegraphics[width=2in,height=1.5in]{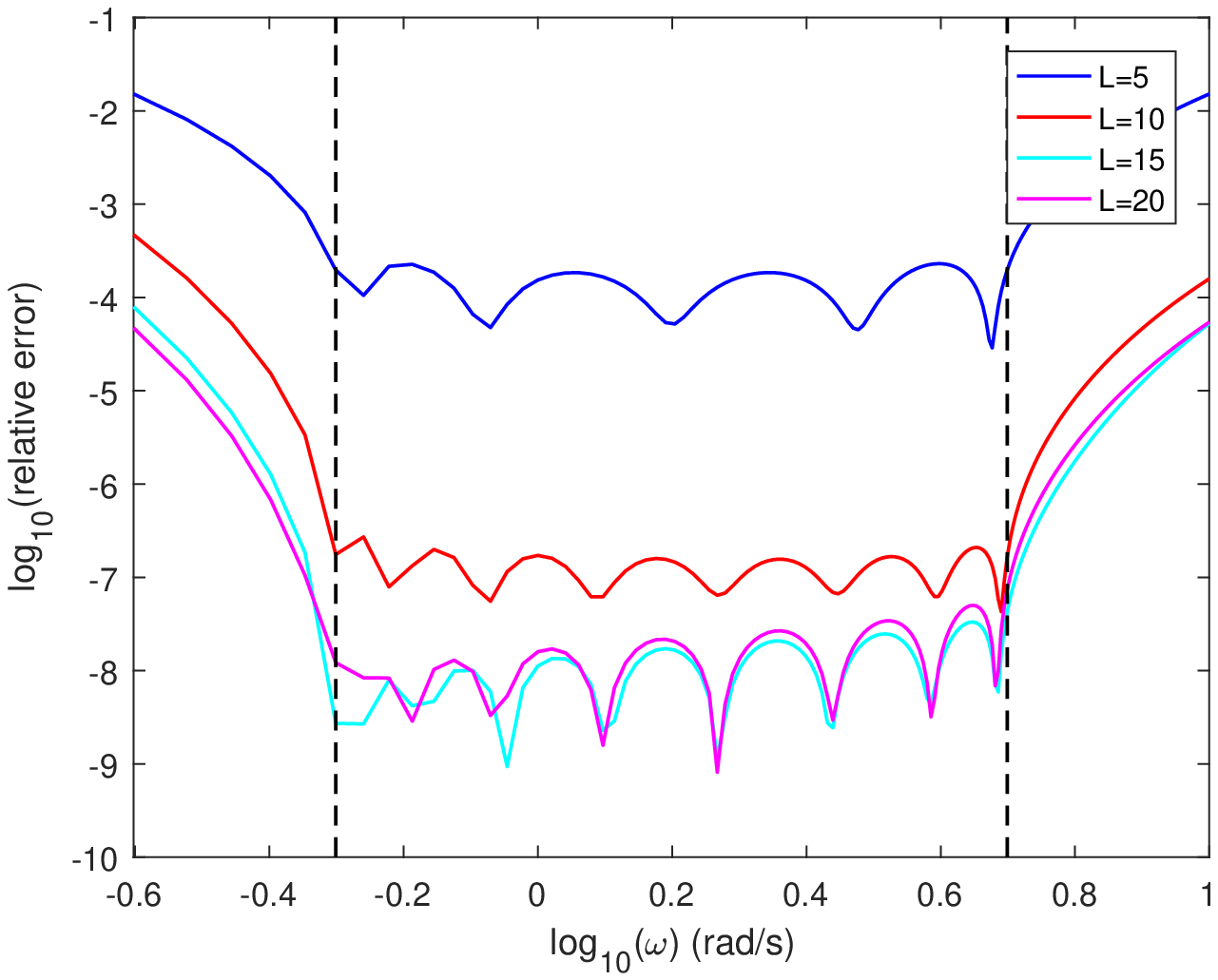}}\\
\subfloat[$M=L$]{%
   \includegraphics[width=2in,height=1.5in]{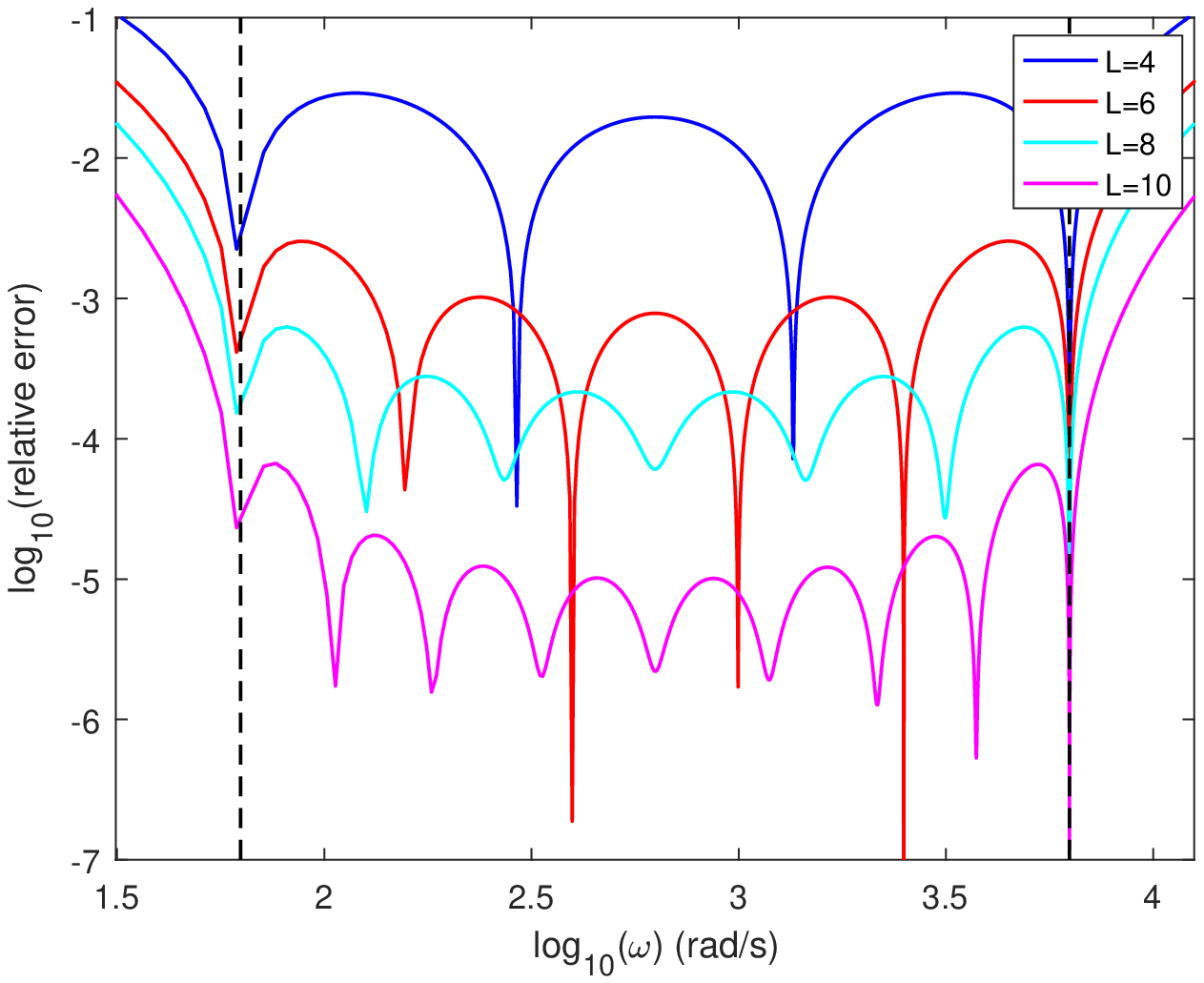}}
\subfloat[$M=2L$]{%
   \includegraphics[width=2in,height=1.5in]{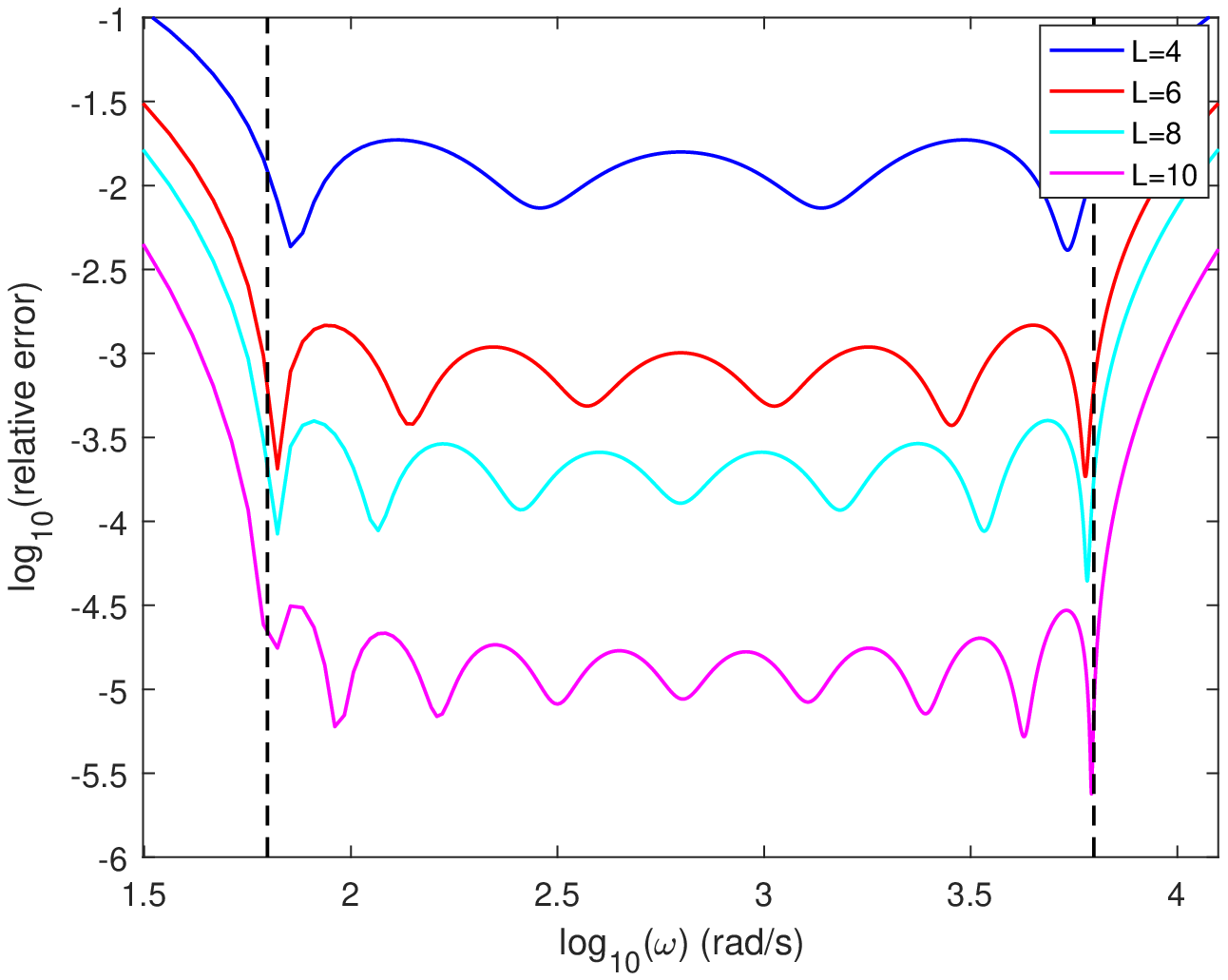}}
\subfloat[$M=3L$]{%
   \includegraphics[width=2in,height=1.5in]{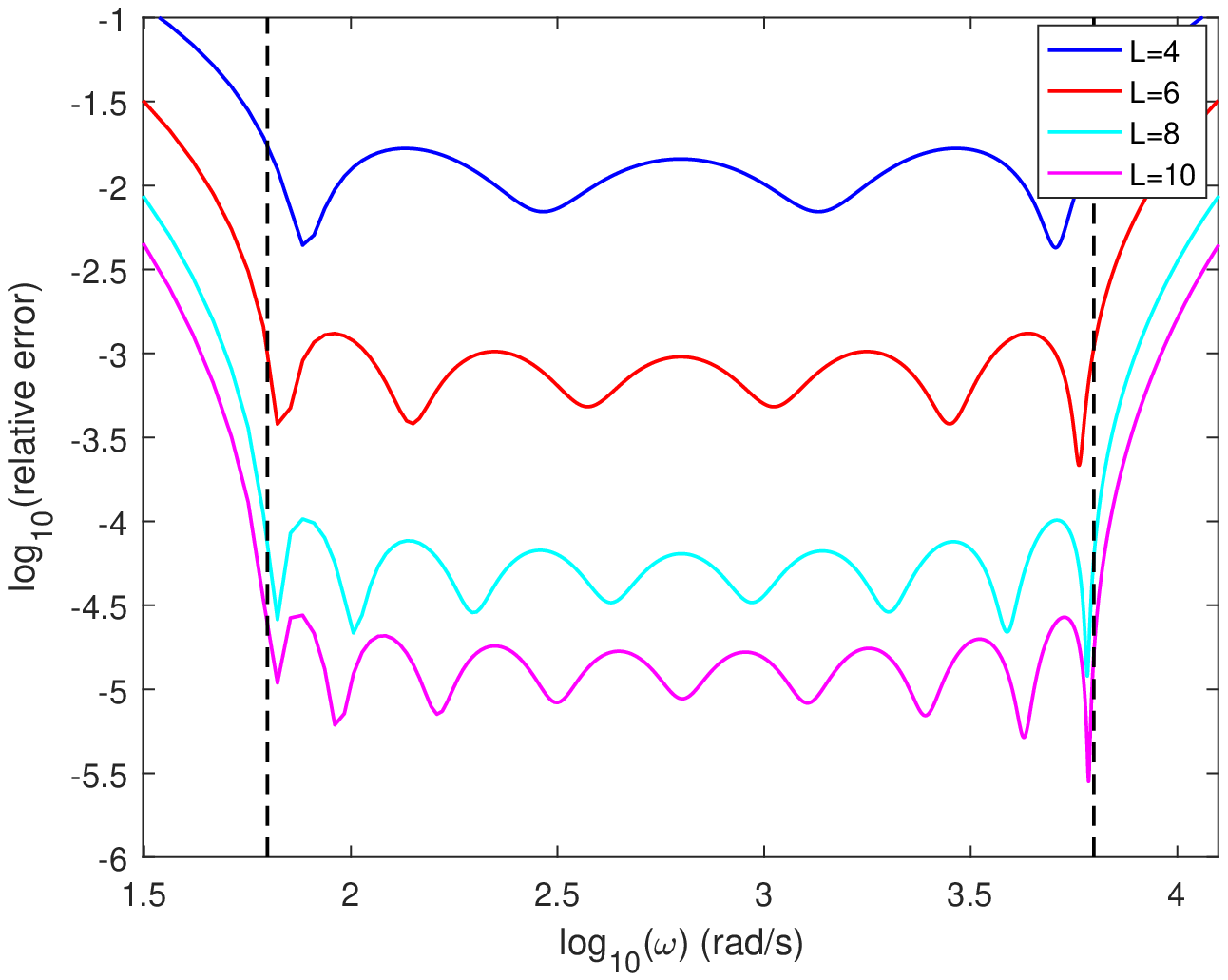}}
\caption{Relative errors of the square system ($M=L$) and overdetermined system ($M>L$), (a)-(c): results for $[0.5,\,5]$, (d)-(f): results for $[20\pi,\,2000\pi]$.  }
\label{fig:relativeerror}
\end{figure}
\begin{figure}[htb]
\centering
\subfloat[$\alpha=0.3$]{%
   \includegraphics[width=2in,height=1.5in]{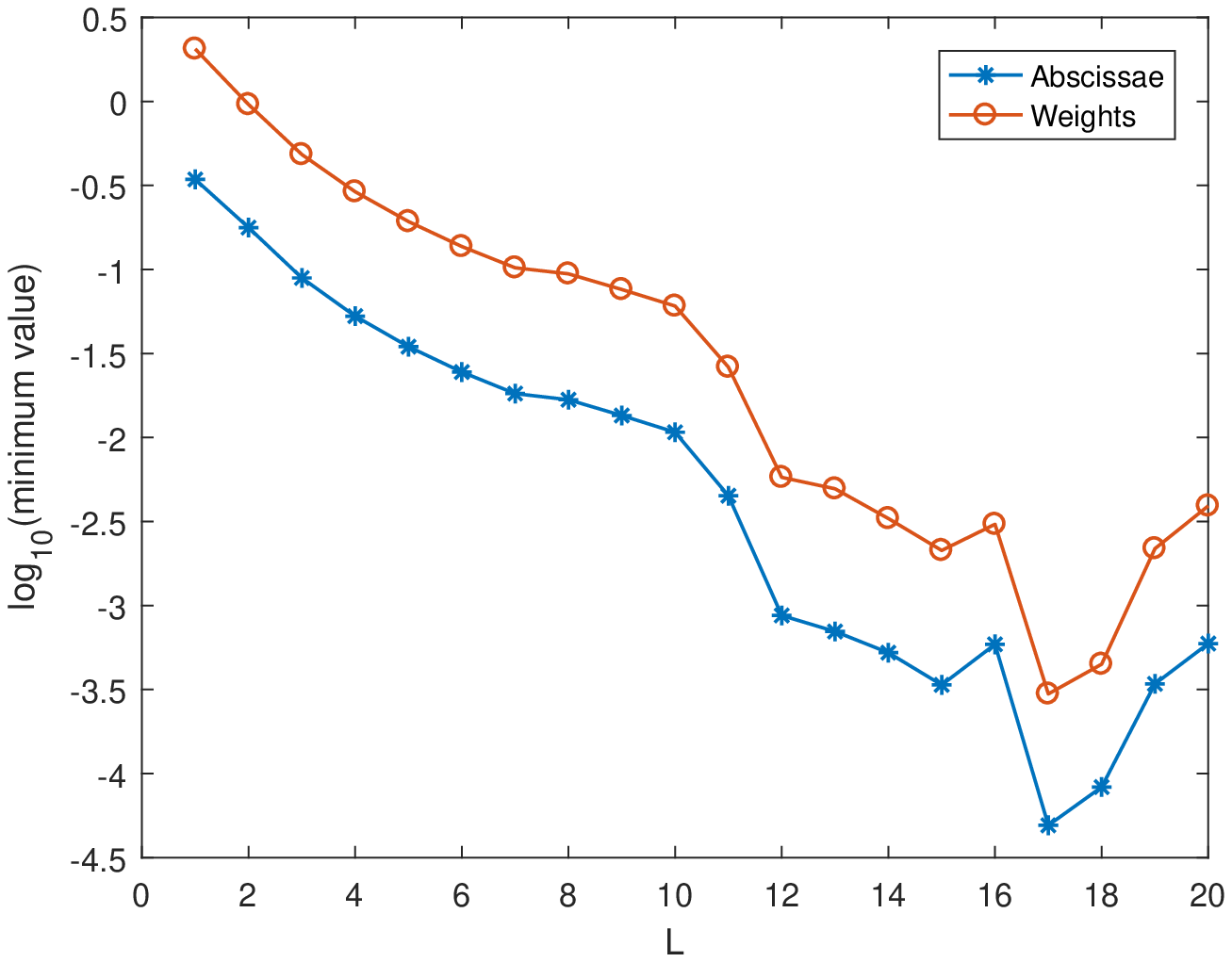}}
\subfloat[$\alpha=0.5$]{%
   \includegraphics[width=2in,height=1.5in]{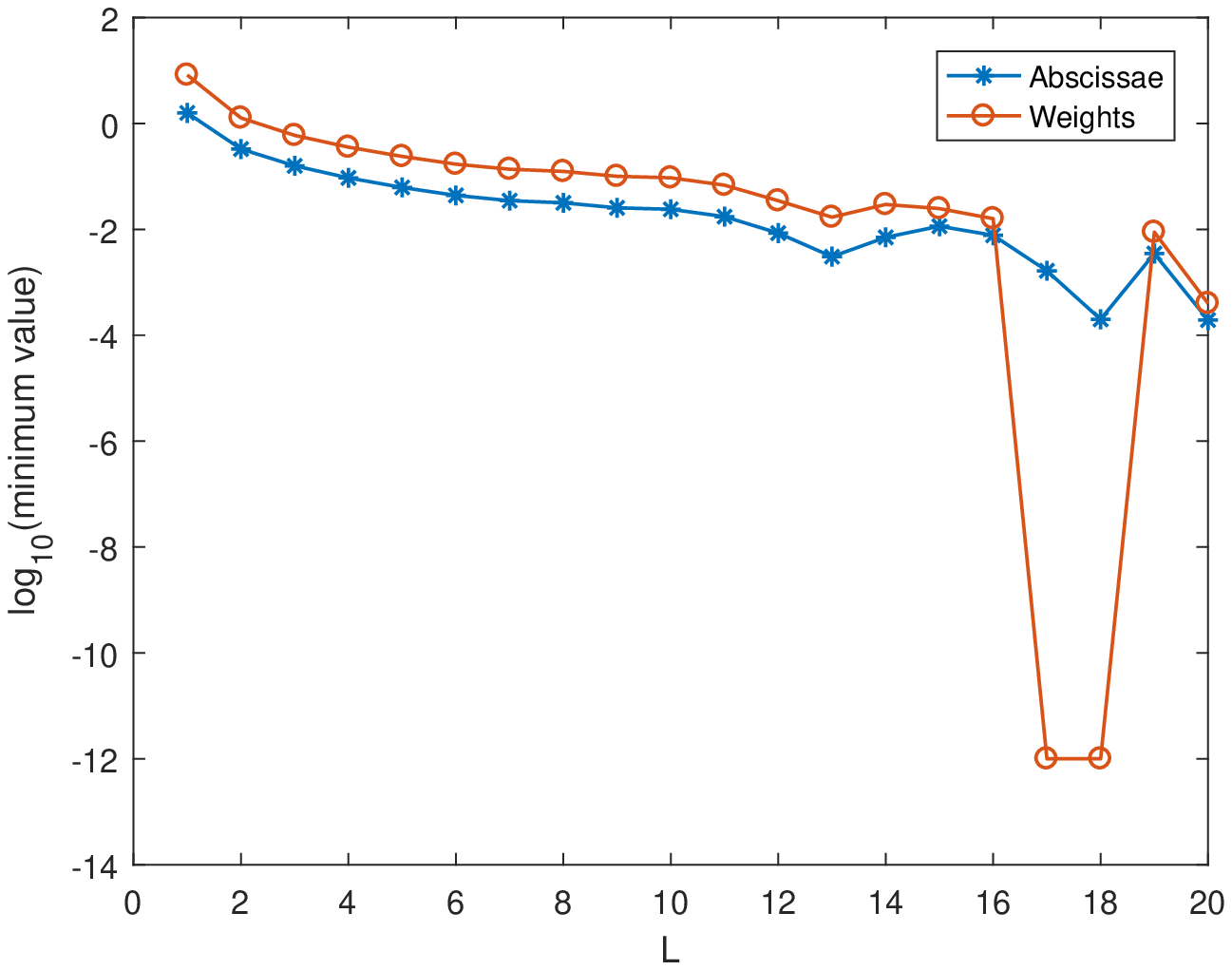}}
\subfloat[$\alpha=0.7$]{%
   \includegraphics[width=2in,height=1.5in]{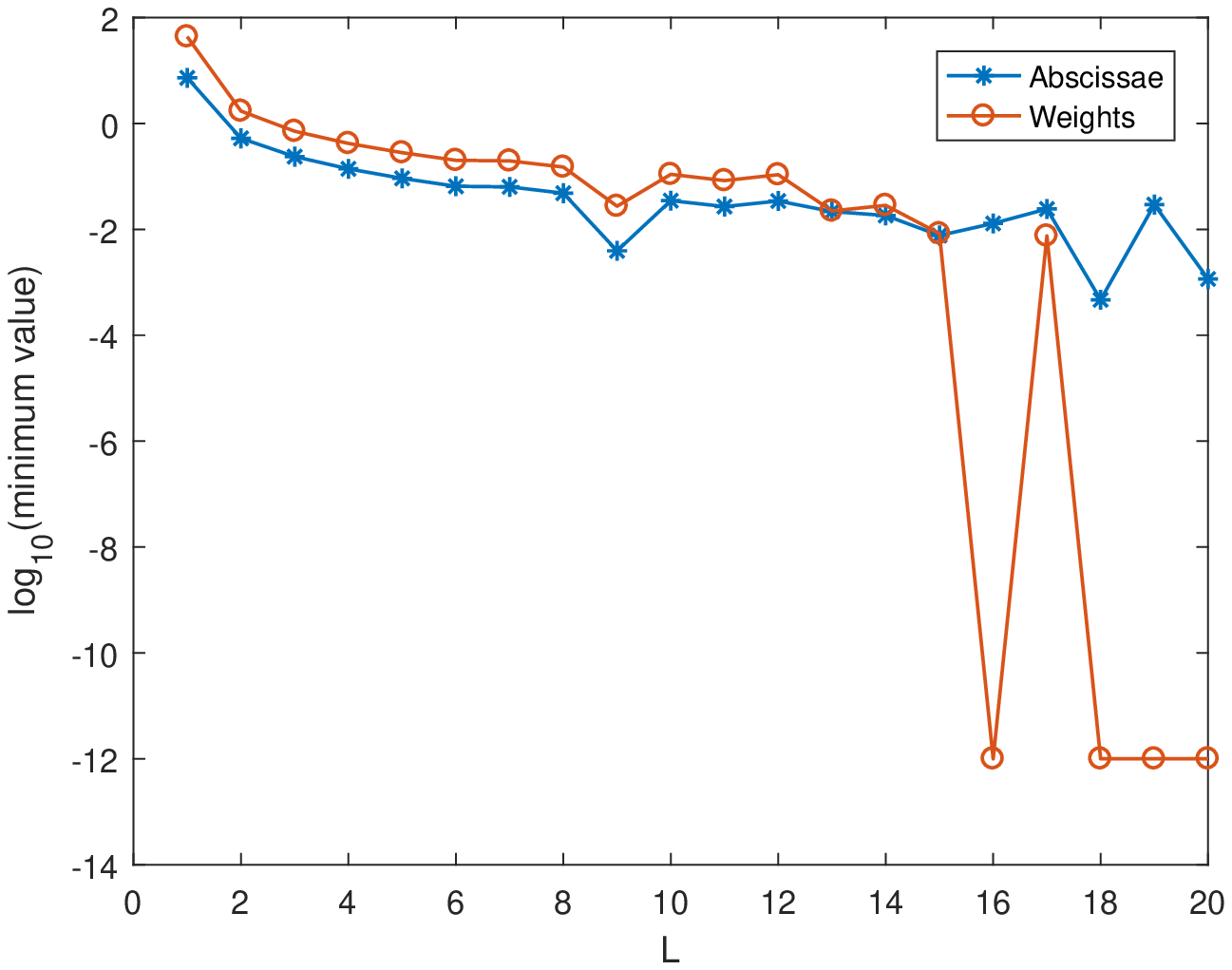}}\\
\subfloat[$\alpha=0.3$]{%
   \includegraphics[width=2in,height=1.5in]{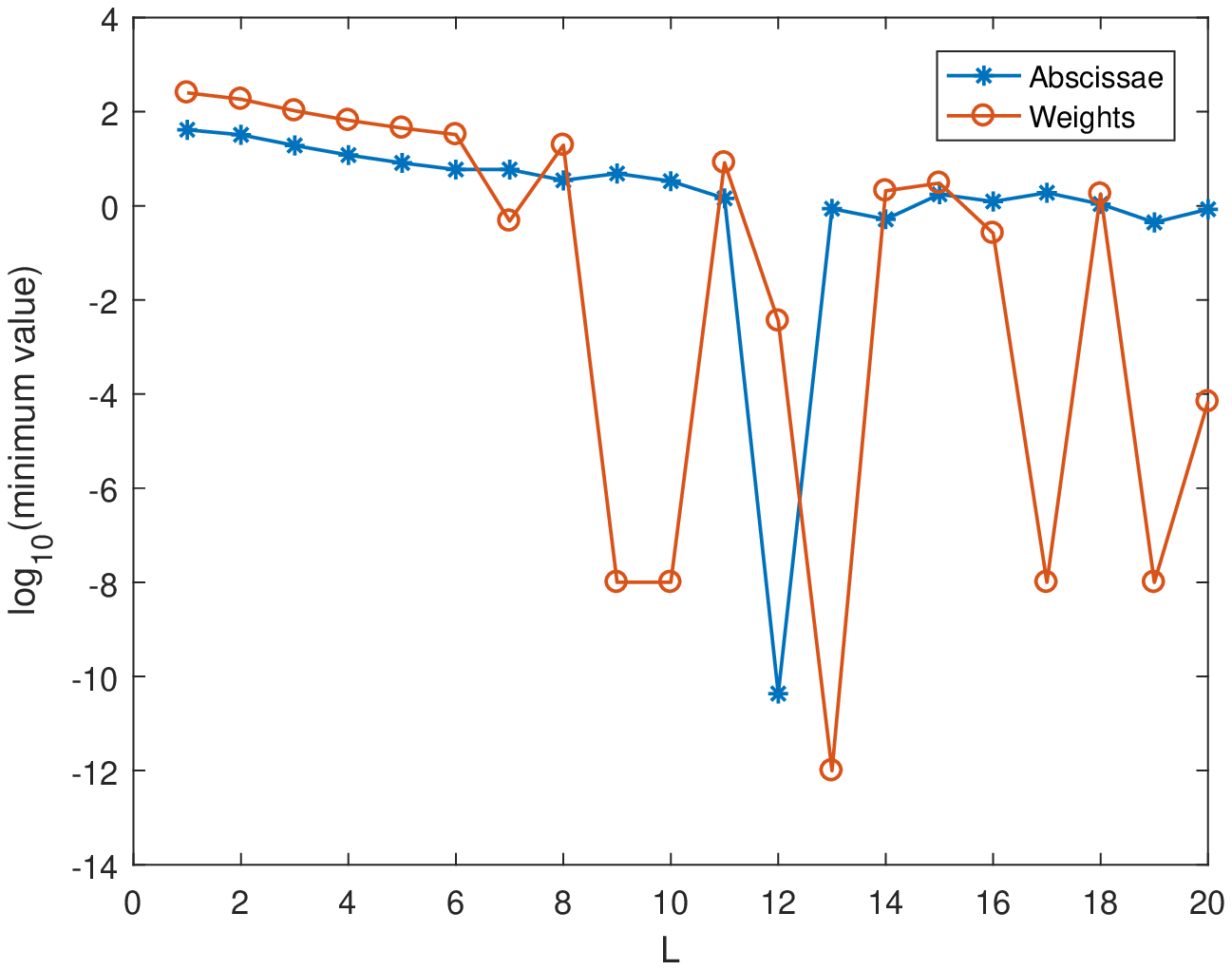}}
\subfloat[$\alpha=0.5$]{%
   \includegraphics[width=2in,height=1.5in]{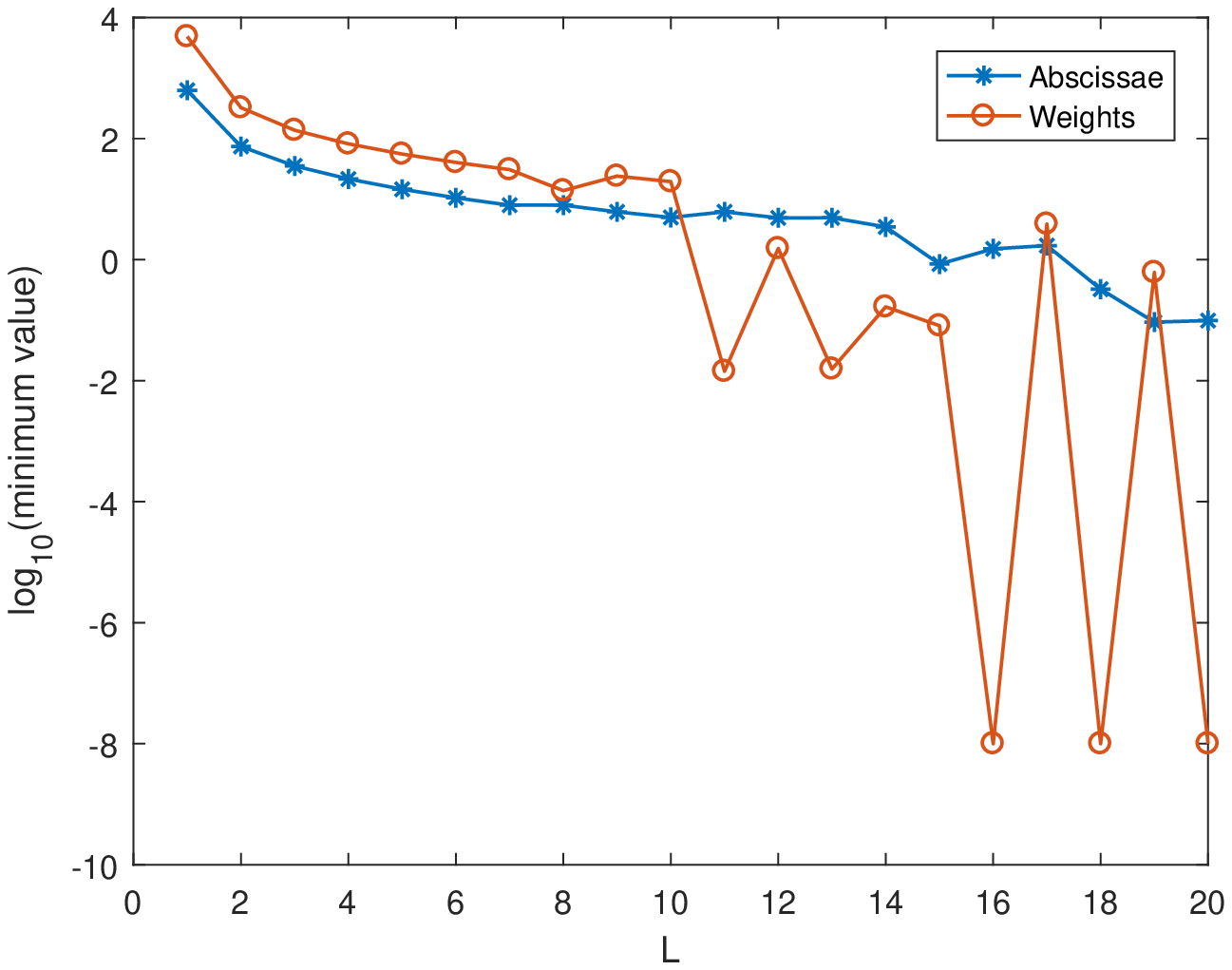}}
\subfloat[$\alpha=0.7$]{%
   \includegraphics[width=2in,height=1.5in]{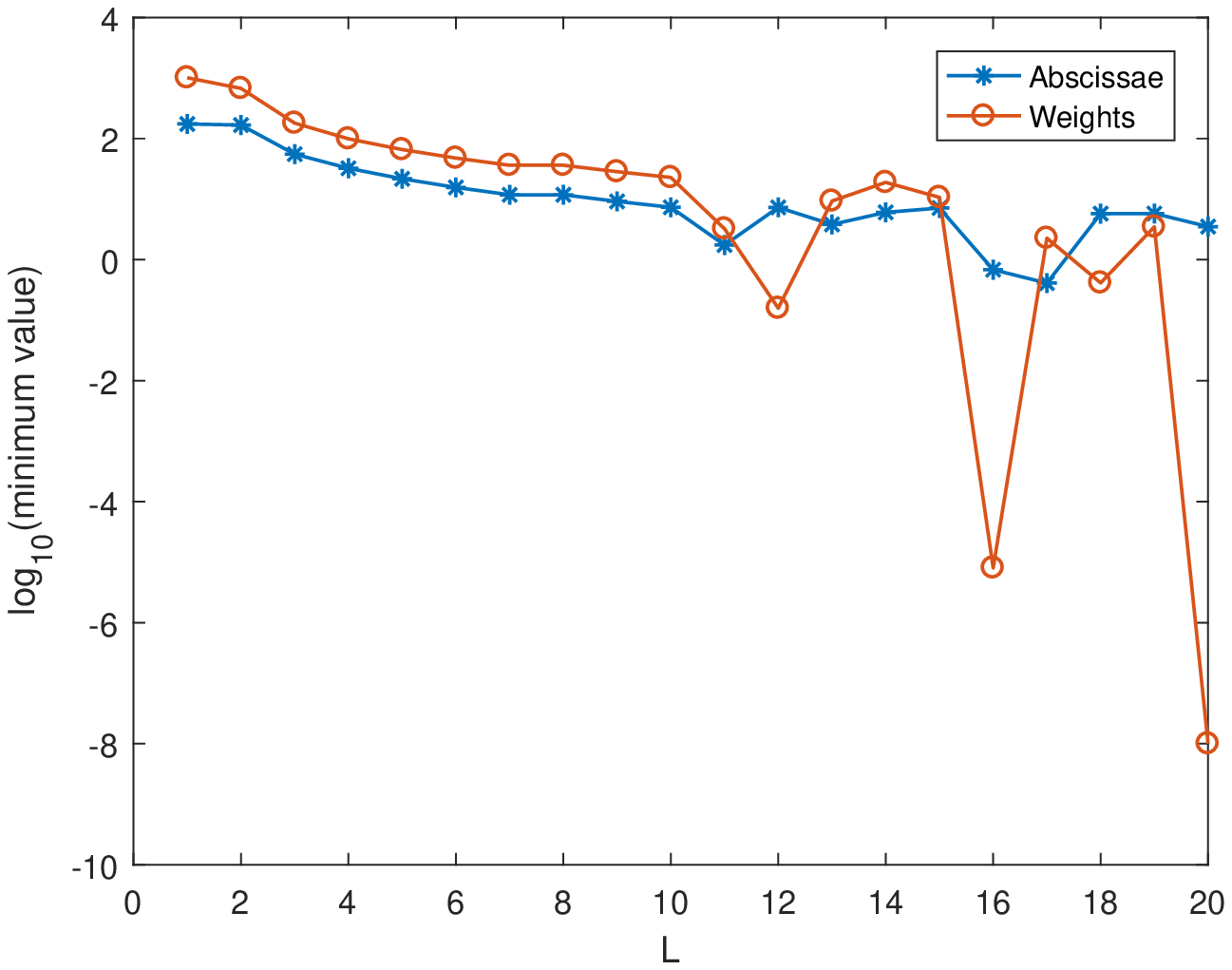}}
\caption{Positivity of weights ($\zeta_\ell$) and abscissae ($\lambda_\ell$) with varying $\alpha$ and $L$, (a)-(c): results for $[0.5,\,5]$, (d)-(f): results for $[20\pi,\,2000\pi]$.  }
\label{fig:positivesign}
\end{figure}

\section{Dispersion analysis}
\label{dispersion}
In this section, we present  the dispersion analysis of the Cole-Cole model (\ref{eq:MaxColeColeH})-(\ref{eq:MaxColeColeP}) and  the approximate  (\ref{eq:MaxColeColeHapp})-(\ref{eq:MaxColeColepsidapp}), where the external source is set to be zero. For this purpose, we assume that the  magnetic field $H$,
electric field $E$ and polarization field $P$
have the plane wave form
\begin{eqnarray*}
H&=&H_0\exp\left(i(k\, x-\omega t)\right),\\
E&=&E_{0}\exp\left(i( {k}\, x-\omega t)\right),\\
P&=&P_{0}\exp\left(i( {k}\, x -\omega t)\right),
\end{eqnarray*}
where $E_{0},\,H_0,\,P_{0}$  are the constant amplitudes, $\omega$  the angular frequency, and $k$ is the wave  number. Injecting $H$ and $E$ into (\ref{eq:MaxColeColeH}) gives
\begin{equation}
\label{eq:disH}
-\omega\mu_0 H_0=k E_{0},
\end{equation}
Substituting $H,\,E,\,P$ to  (\ref{eq:MaxColeColeE}) leads to
\begin{equation}
-\omega\epsilon_0\epsilon_{\infty}E_0=k H_0
+\omega P_0.
\end{equation}
According to  the polarization equation (\ref{eq:MaxColeColeP}), we have
\begin{equation}\label{eq:disPananoapp}
\left(1+\mathcal{Q}(\omega)\right)P_0=\epsilon_0
\left(\epsilon_s-\epsilon_{\infty}\right)E_0,\quad \mathcal{Q}(\omega)=(i\omega\tau_0)^{\alpha}.
\end{equation}
Combining \eqref{eq:disH}-(\ref{eq:disPananoapp}), we have
\begin{equation}\label{eq:disPanaeig}
k^2=
\left(\mu_0\epsilon_0\epsilon_{\infty}+\frac{\mu_0\epsilon_0
\left(\epsilon_s-\epsilon_{\infty}\right)}
{1+\mathcal{Q}(\omega)}\right)\omega^2.
\end{equation}
In the above equation, if we replace $\mathcal{Q}(\omega)$ by
$$
\mathcal{Q}^{\sharp}(\omega)=\tau_0^{\alpha}\frac{\sin(\pi\alpha)}{\pi}i\omega\sum_{\ell=1}^L
\frac{\zeta_\ell\lambda_\ell^{\alpha-1}}{i\omega+\lambda_\ell},
$$
 we can obtain the dispersion relations of the approximate   system  (\ref{eq:MaxColeColeHapp})-(\ref{eq:MaxColeColepsidapp}).

With the wave modes given by (\ref{eq:disPanaeig}), we deduce the phase velocities $c=\omega/\Re(k)$ and derive the attenuations $\eta=-\Im(k )$, where $\Re{k}$ and $\Im{k}$ are the real and imaginary parts of $k$, respectively. In Fig.\,\ref{fig:DispConcurv}, we illustrate the influence of $\alpha$ on the phase velocity and the attenuation, with the parameters  $\tau_0=\mu_0=\epsilon_0=\epsilon_p=\epsilon_{\infty}=1$.
Note that the case of $\alpha=1$ corresponds to the Debye model. In  addition, we use the same setting to demonstrate the  dispersion relations of the approximate system  (\ref{eq:MaxColeColeHapp})-(\ref{eq:MaxColeColepsidapp}). In Fig.\,\ref{fig:dispersioncurveapp}, we  present the comparison of the phase velocities and attenuations obtained from the exact $\mathcal{Q}(\omega)$ and the approximate $\mathcal{Q}^{\sharp}(\omega)$; it shows an excellent agreement between the exact and approximate results. Here the weights $\zeta_\ell$ and abscissae $\lambda_\ell$ are obtain by the non-linear optimization process (section \ref{optimization}) with $[20\,\pi,\,200\,\pi]$ and $L=6$.
\begin{figure}[htb]
\centering
\subfloat[phase velocity]{%
   \includegraphics[width=2.5in,height=2.0in]{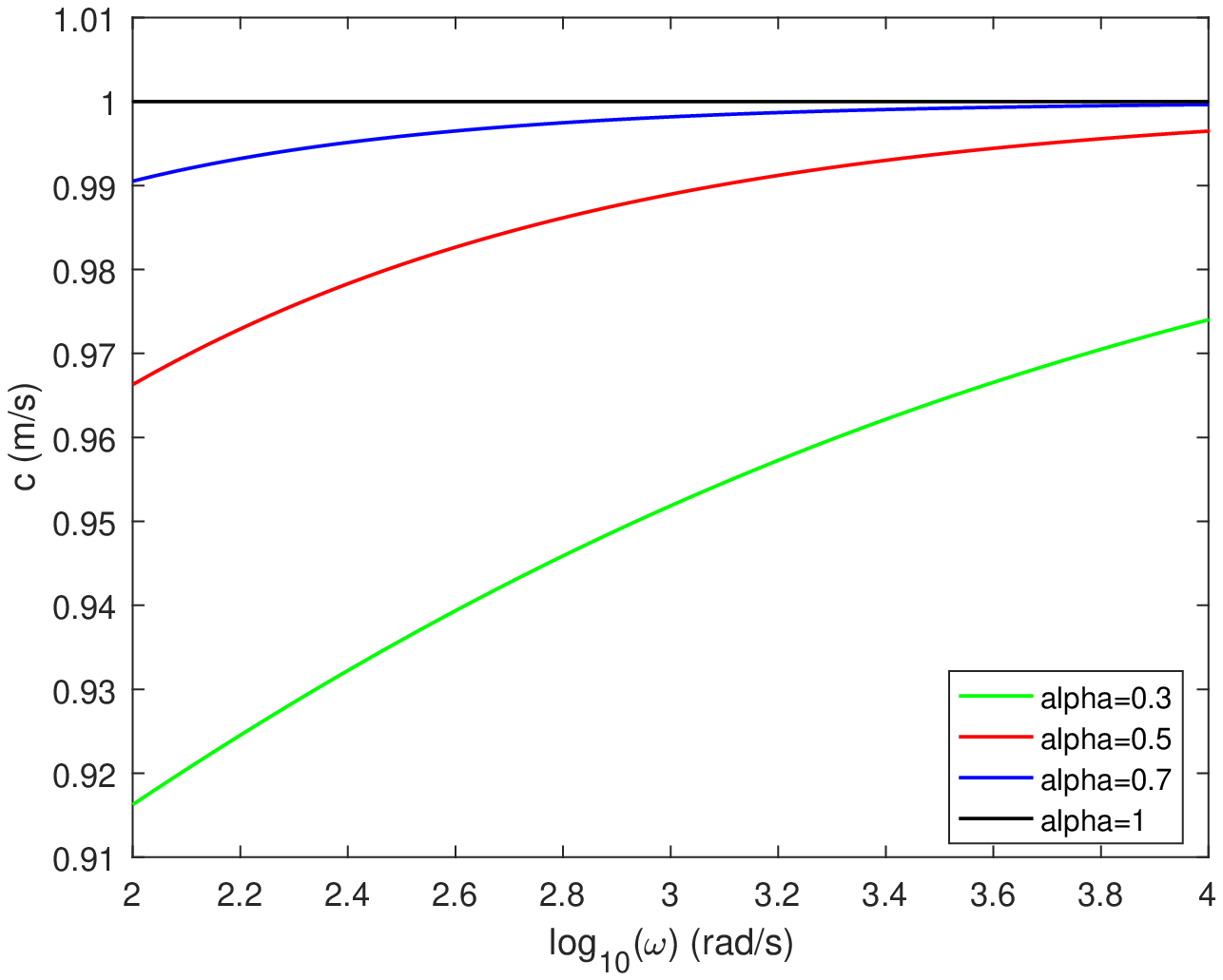}}
\subfloat[attenuation]{%
   \includegraphics[width=2.5in,height=2.0in]{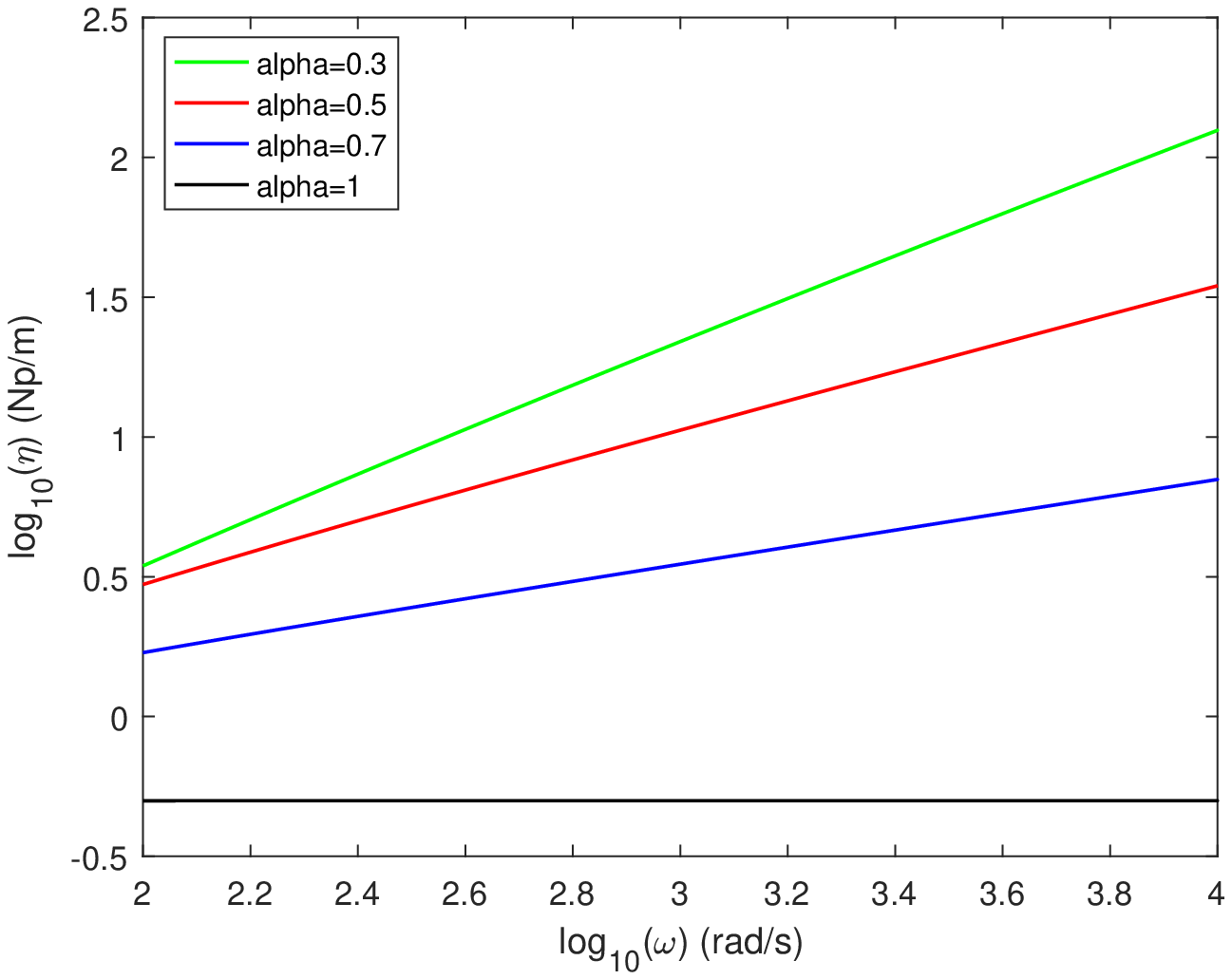}}
\caption{Dispersion relations of the Cole-Cole  model (\ref{eq:MaxColeColeH})-(\ref{eq:MaxColeColeP})  with  varying $\alpha=0.3,\,0.5,\,0.7,\,1$.  }
\label{fig:DispConcurv}
\end{figure}
\begin{figure}[htb]
\centering
\subfloat[ $\alpha=0.3$]{%
   \includegraphics[width=2in,height=1.5in]{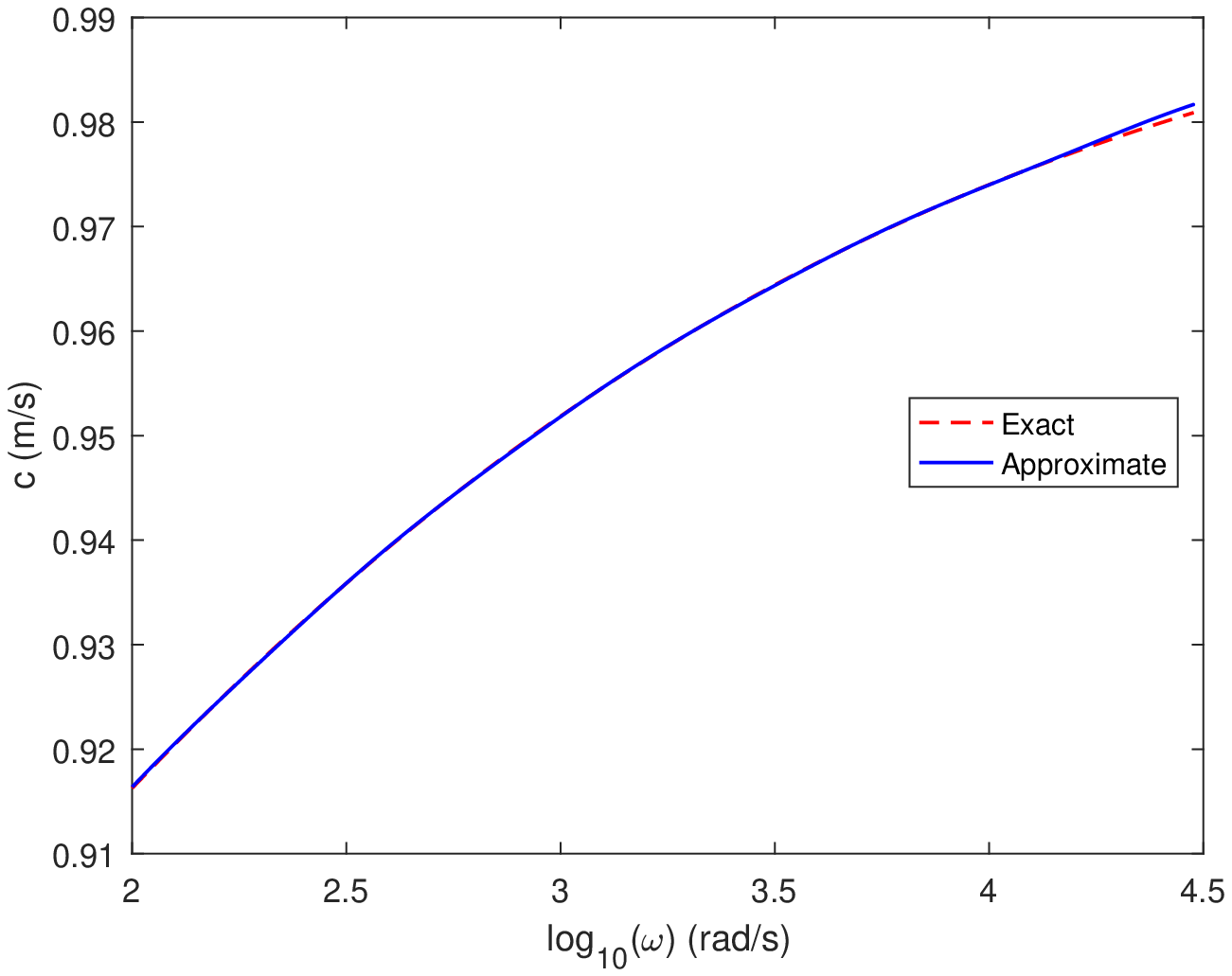}}
\subfloat[ $\alpha=0.5$]{%
   \includegraphics[width=2in,height=1.5in]{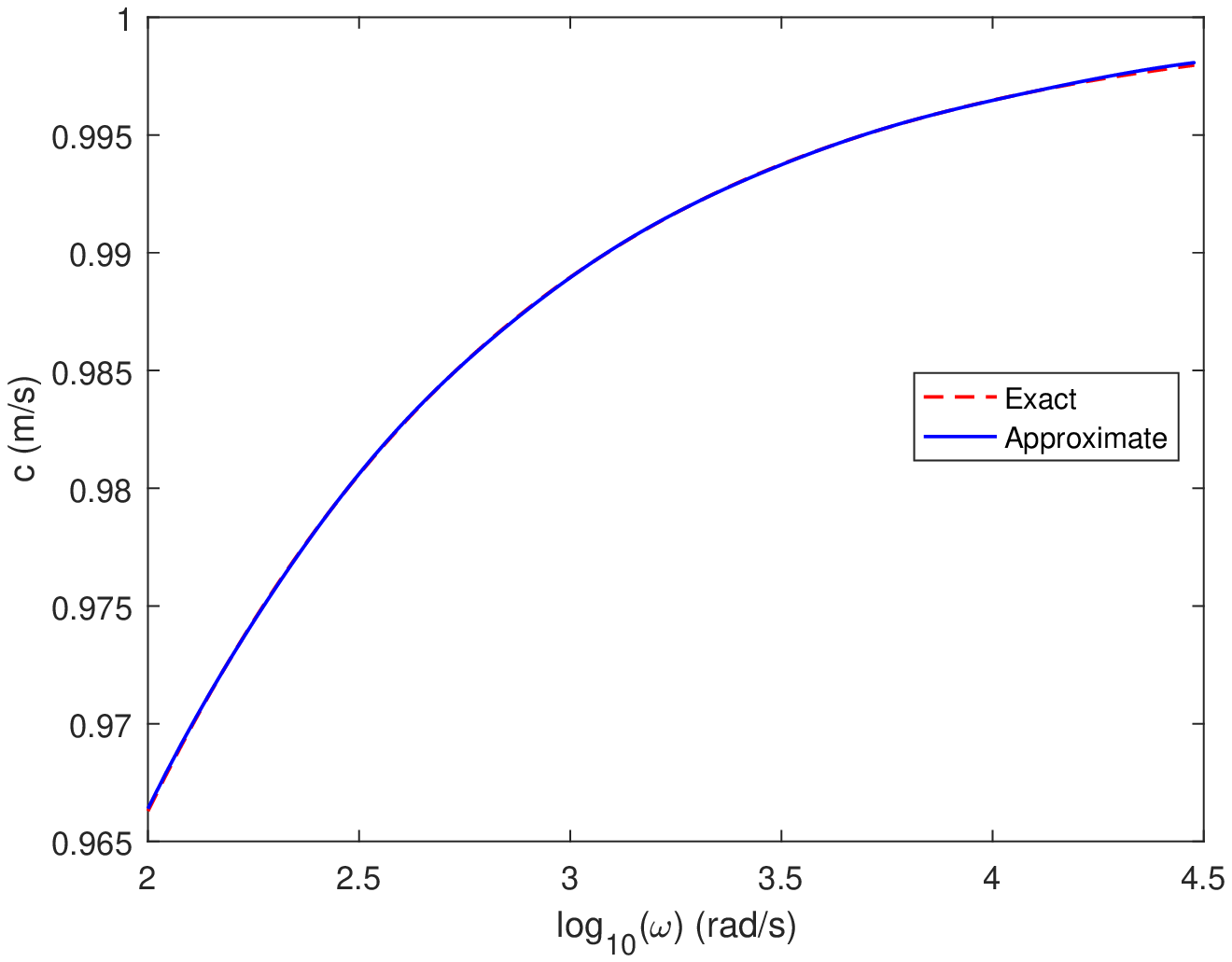}}
\subfloat[$\alpha=0.7$]{%
   \includegraphics[width=2in,height=1.5in]{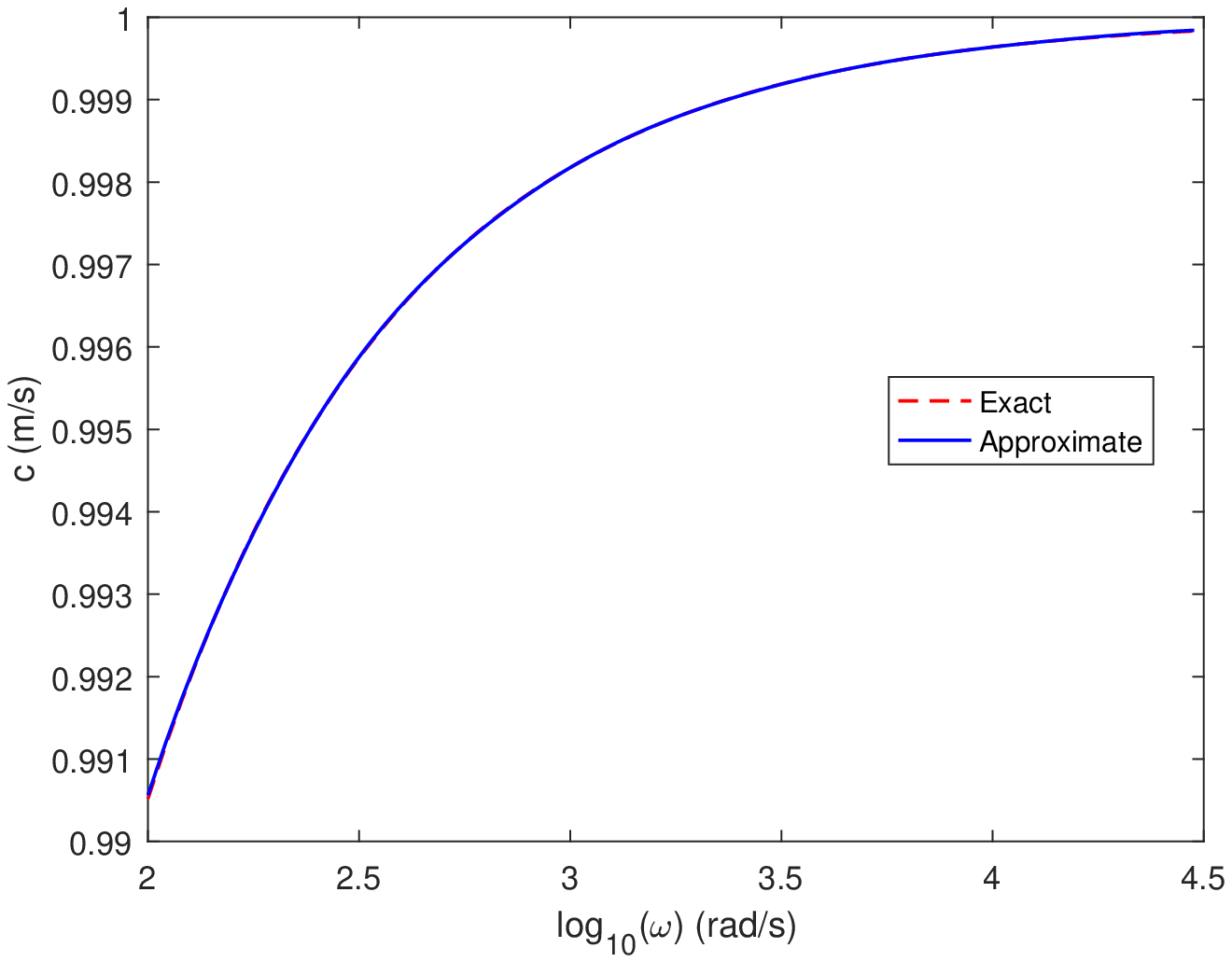}}\\
\subfloat[$\alpha=0.3$]{%
   \includegraphics[width=2in,height=1.5in]{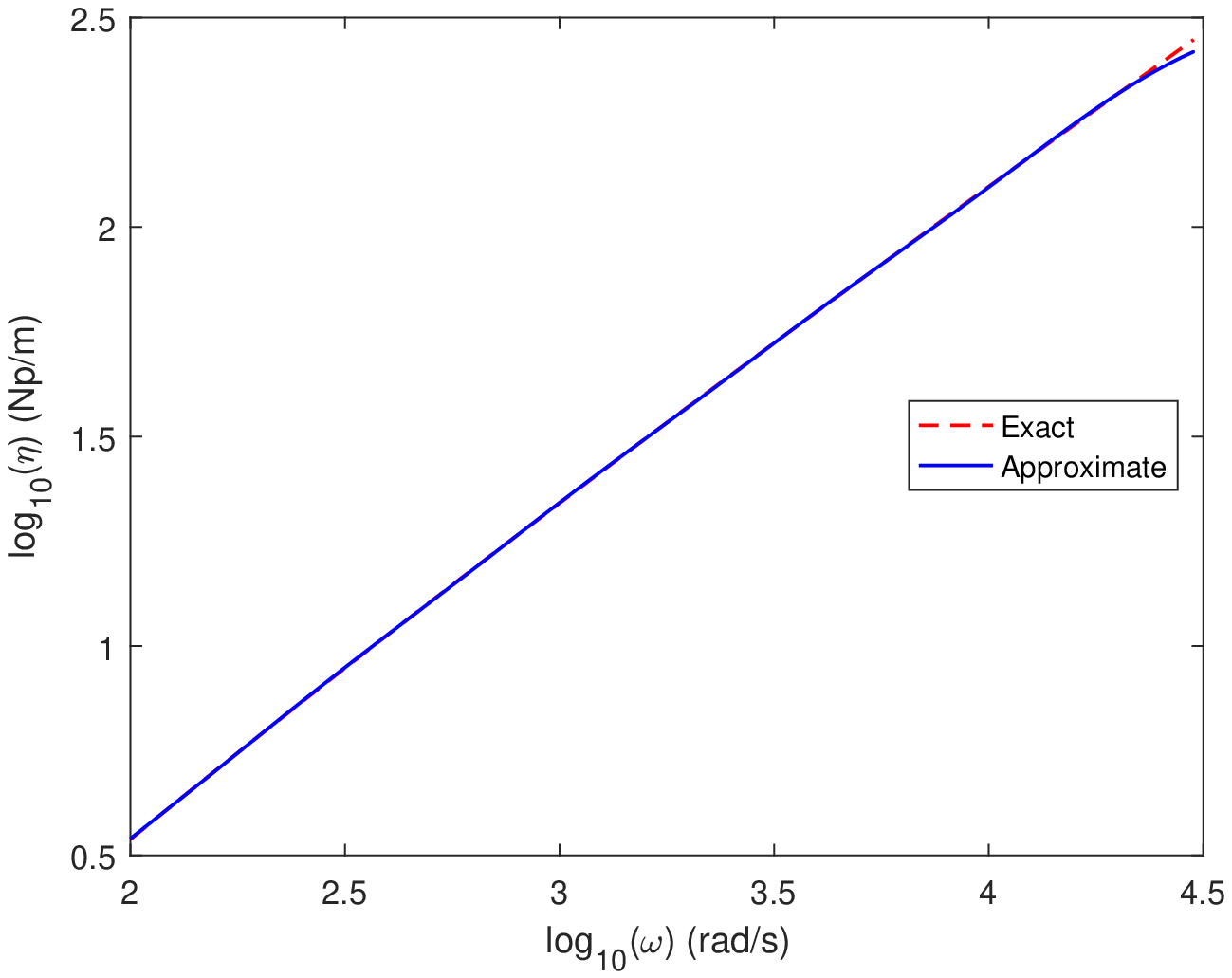}}
\subfloat[$\alpha=0.5$]{%
   \includegraphics[width=2in,height=1.5in]{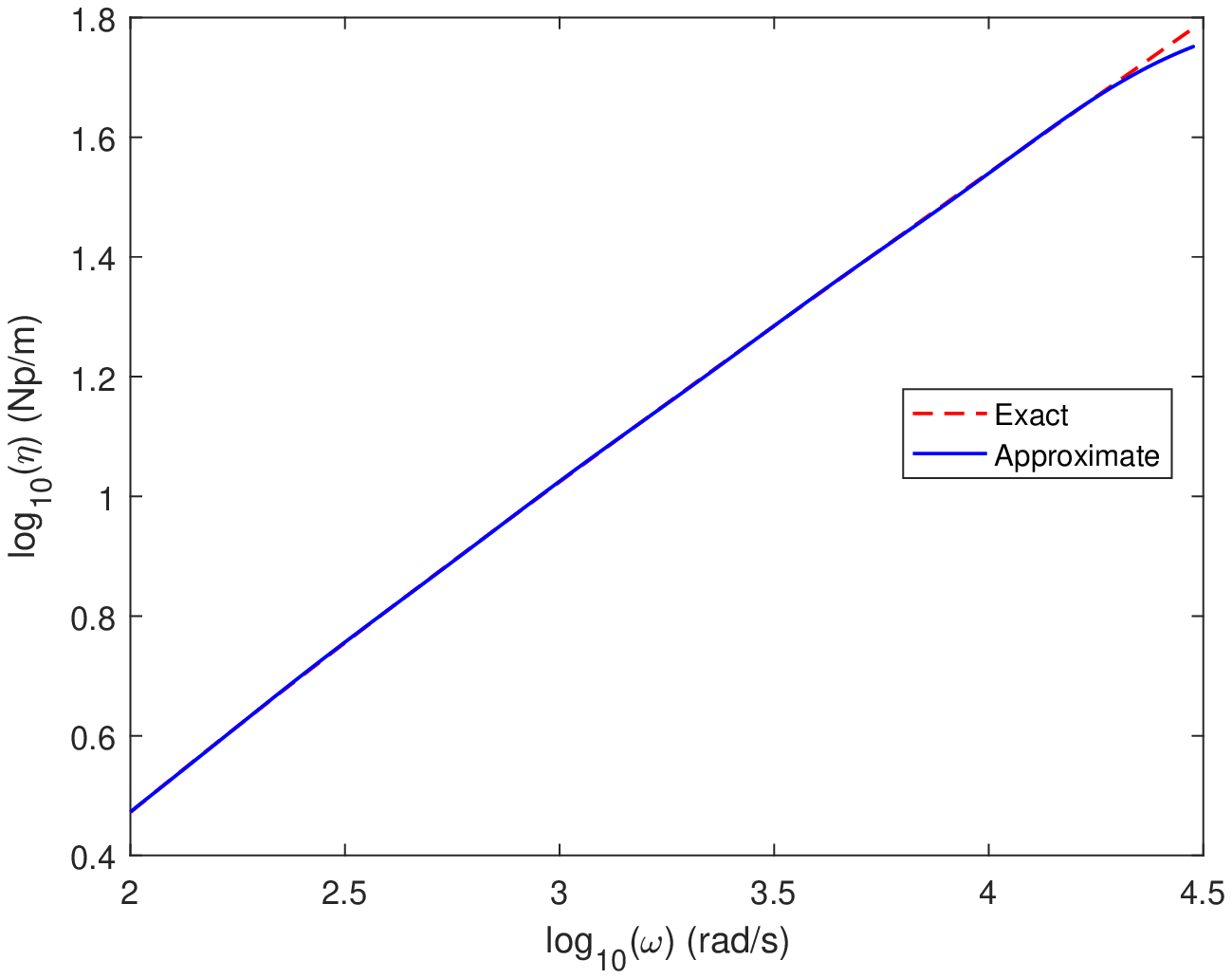}}
\subfloat[$\alpha=0.7$]{%
   \includegraphics[width=2in,height=1.5in]{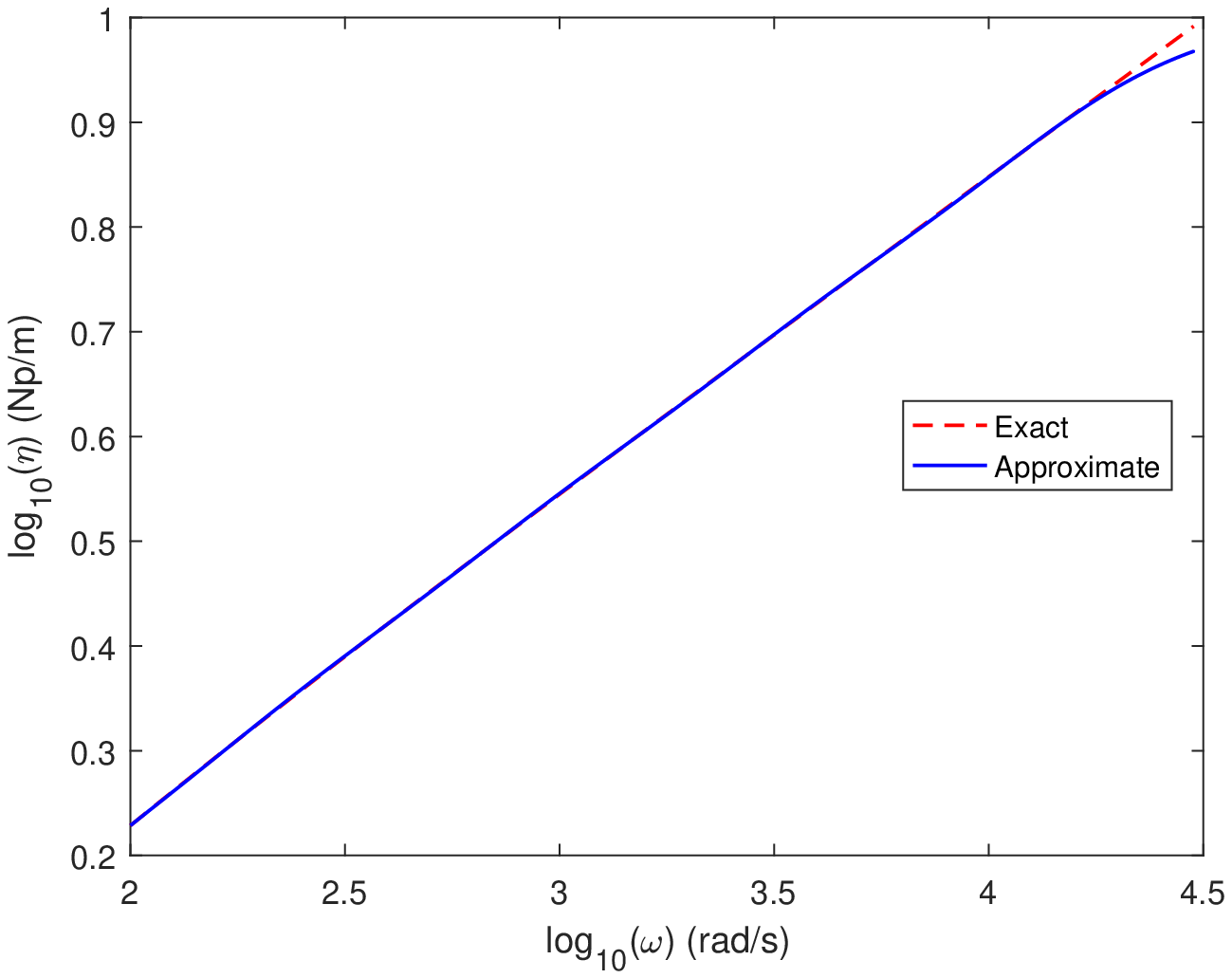}}
\caption{Dispersion relations of the approximate Cole-Cole  model  (\ref{eq:MaxColeColeHapp})-(\ref{eq:MaxColeColepsidapp}) with  varying $\alpha=0.3,\,0.5,\,0.7$, (a)-(c): phase velocity, (d)-(f): attenuation.}
\label{fig:dispersioncurveapp}
\end{figure}

\section{Numerical examples}
\label{numerical}
In this section, numerical examples are presented  to demonstrate the accuracy of the proposed algorithm. For simplicity, we set $\tau_0=\mu_0=\epsilon_0\epsilon_{\infty}=\epsilon_0(\epsilon_s-\epsilon_{\infty})=1$. To our best knowledge, the regularity theory of the solutions to \eqref{eq:MaxColeColeH}-\eqref{eq:iniMax} with zero external sources is still open. 
The source terms $F_i,\,i=1,2,3$  are added to  assume that the exact solutions are sufficiently smooth  and $P(x,t)$ have many vanishing derivatives at the initial time.
The time discretization is achieved by  BDF2 formula. More precisely, given initial data $U^0,\,U^1$ with $U=H,\,E,\,P$, we  have from the auxiliary equation  (\ref{eq:MaxColeColepsidapp})
$$
C_1\psi_\ell^n+C_2\psi_\ell^{n-1}+C_3\psi_\ell^{n-2}=-\lambda_\ell\psi_\ell^n
+C_{\alpha}\lambda_\ell^{\alpha-1}\left(C_1P^n+C_2P^{n-1}+C_3P^{n-2}\right),
$$
or equivalently
\begin{equation*}\label{eq:psinn}
\psi_\ell^n=\frac{C_{\alpha}}{C_1+\lambda_\ell}\lambda_\ell^{\alpha-1}\left(C_1P^n +C_2P^{n-1}+C_3P^{n-2}  \right)
-\frac{C_{\alpha}}{C_1+\lambda_\ell}\left(C_2\psi_\ell^{n-1}+C_3\psi_\ell^{n-2}\right),
\end{equation*}
where we have set
$$
C_1=\frac{3}{2\tau},\quad C_2=-\frac{2}{\tau},\quad C_3=\frac{1}{2\tau},\quad C_{\alpha}=\frac{\sin(\pi \alpha)}{\pi}.
$$
Injecting $\psi_\ell^n$ into (\ref{eq:MaxColeColePapp}), we derive
\begin{equation*}
\begin{aligned}
\left(1+\tau_0^{\alpha}C_1C_{\alpha}\sum_{\ell=1}^L
\frac{\zeta_\ell\lambda_\ell^{\alpha-1}}{C_1+\lambda_\ell}\right)P^n=&E^n-\tau_0^{\alpha}C_{\alpha}
\sum_{\ell=1}^L
\frac{\zeta_\ell\lambda_\ell^{\alpha-1}}{C_1+\lambda_\ell}\left(C_2P^{n-1}+C_3P^{n-2}\right)\\
&+\tau_0^{\alpha}\sum_{\ell=1}^L\frac{\zeta_\ell}{C_1+\lambda_\ell}\left(C_2\psi_\ell^{n-1}
+C_3\psi_\ell^{n-2}\right).
\end{aligned}
\end{equation*}
Using the standard BDF2 formula for (\ref{eq:MaxColeColeHapp}) and (\ref{eq:MaxColeColeEapp}) and combining $P^n$ give the computation scheme.
The first initial data are set to be   $U_0,\,U=H,\,E,\,P$ and the second initial data are obtained by using forward Euler formula to (\ref{eq:MaxColeColeHapp})-(\ref{eq:MaxColeColepsidapp}).

\subsection{Accuracy test with two external sources\label{test1}}
As the first example, we aim  to demonstrate the convergence rates of the DG scheme.
For this purpose, we consider the case of $F_3=0$ and
\begin{eqnarray*}
F_1&=&\pi\sin(\pi x)\left(\frac{2\,t^{2-\alpha}}{\Gamma(1-\alpha)(1-\alpha)(2-\alpha)}+t^2\right)
+2\pi\left(2\cos(\pi x)+\sin(\pi x)\right)\,t,\\
F_2&=&
\cos(\pi x)\left(\frac{2\, t^{1-\alpha}}{\Gamma(1-\alpha)(1-\alpha)}+4\,t \right)
-\pi^2\left(-2\sin(\pi x)+\cos(\pi x)\right)\,t^2.
\end{eqnarray*}
For this case, it can be checked that 
 (\ref{eq:MaxColeColeH})-(\ref{eq:MaxColeColeP})  are satisfied by the following exact solution
\begin{eqnarray*}
E(x,t)&=&\cos(\pi x)\left(\frac{2\, t^{2-\alpha}}{\Gamma(1-\alpha)(1-\alpha)(2-\alpha)}+t^2\right), \\
H(x,t)&=&\pi\,(2\cos(\pi x)+\sin(\pi x))\,t^2, \\
P(x,t)&=&\cos(\pi x)\,t^2.
\end{eqnarray*}

The initial data are obtained by setting $t=0$ of the  exact solutions and the periodic boundary condition is used. The computation domain is $\Omega=[0,\, 2]$ and the total simulation time is $T=2$.
The frequency band is $[0.5,\, 5]$ and the number of quadrature points is $L=20$. Moreover,
we set  $\tau=h^2$ for both $\mathbb{P}^1$ and $\mathbb{P}^2$ elements such that the temporal discretization error can be relatively negligible.
In Tables \ref{ex3alpha03}-\ref{ex3alpha07}, we report the spatial errors in the  $L^2$-norm and convergence orders with varying $\alpha$. The expected rates of  second order for $\mathbb{P}^1$ elements and third order for $\mathbb{P}^2$ elements are reached, which  confirm the optimal accuracy order of the proposed scheme.
\begin{table}
\centering
\begin{tabular}{p{0.5pt} p{5pt}  p{40pt} p{12pt} p{0.1pt} p{40pt} p{12pt} p{0.1pt} p{40pt} p{12pt}}
\hline
   \multirow{2}{*}{}&\multirow{2}{*}{ $N_x$} & \multicolumn{2}{c}{E}& & \multicolumn{2}{c}{H}& &\multicolumn{2}{c}{P} \\
\cline{3-4} \cline{6-7} \cline{9-10}
  &   & Error &Order &    & Error &Order &  & Error &Order\\
\hline
   \multirow{4}{*}{\ $\mathbb{P}^1$}
    &10 &1.403E-1  & --    &  & 5.118E-1  & --     & & 6.784E-2  &--    \\
    &20 &3.455E-2  & 2.022 &  & 1.220E-1  & 2.069  & & 1.678E-2  &2.015   \\
    &40 &8.553E-3  & 2.014 &  & 2.972E-2  & 2.037  & & 4.157E-3  &2.013   \\
    &80 &2.127E-3  & 2.008 &  & 7.330E-3  & 2.020  & & 1.034E-3  &2.007    \\
\hline
   \multirow{4}{*}{\ $\mathbb{P}^2$}
    &10 &2.834E-2  & --    &  & 6.605E-3  & --    & & 1.353E-2  &--    \\
    &20 &3.605E-3  & 2.975 &  & 8.868E-4  & 2.897 & & 1.727E-3  &2.970    \\
    &40 &4.535E-4  & 2.991 &  & 1.172E-4  & 2.920 & & 2.178E-4  &2.987   \\
    &80 &5.682E-5  & 2.997 &  & 1.522E-5  & 2.945 & & 2.728E-5  &2.997     \\
\hline
\end{tabular}
\caption{The $L^2$ errors and convergence orders of E,\,H,\,P for Example \ref{test1} with $\alpha=0.3$.}
\label{ex3alpha03}
\end{table}
\begin{table}
\centering
\begin{tabular}{p{0.5pt} p{5pt}  p{40pt} p{12pt} p{0.1pt} p{40pt} p{12pt} p{0.1pt} p{40pt} p{12pt}}
\hline
   \multirow{2}{*}{}&\multirow{2}{*}{ $N_x$} & \multicolumn{2}{c}{E}& & \multicolumn{2}{c}{H}& &\multicolumn{2}{c}{P} \\
\cline{3-4} \cline{6-7} \cline{9-10}
  &   & Error &Order &    & Error &Order &  & Error &Order\\
\hline
   \multirow{4}{*}{\ $\mathbb{P}^1$}
    &10 &1.405E-1  & --    &  & 5.110E-1  & --     & & 6.722E-2  &--    \\
    &20 &3.470E-2  & 2.018 &  & 1.219E-1  & 2.068  & & 1.670E-2  &2.009   \\
    &40 &8.621E-3  & 2.009 &  & 2.972E-2  & 2.036  & & 4.148E-3  &2.009   \\
    &80 &2.151E-3  & 2.003 &  & 7.335E-3  & 2.019  & & 1.033E-3  &2.006    \\
\hline
   \multirow{4}{*}{\ $\mathbb{P}^2$}
    &10 &2.852E-2  & --    &  & 5.938E-3  & --    & & 1.310E-2  &--    \\
    &20 &3.645E-3  & 2.968 &  & 7.834E-4  & 2.922 & & 1.681E-3  &2.962    \\
    &40 &4.599E-4  & 2.987 &  & 1.019E-4  & 2.943 & & 2.124E-4  &2.985   \\
    &80 &5.774E-5  & 2.994 &  & 1.312E-5  & 2.957 & & 2.654E-5  &3.000     \\
\hline
\end{tabular}
\caption{The $L^2$ errors and  convergence orders of E,\,H,\,P for Example \ref{test1} with $\alpha=0.5$.}
\label{ex3alpha05}
\end{table}
\begin{table}
\centering
\begin{tabular}{p{0.5pt} p{5pt}  p{40pt} p{12pt} p{0.1pt} p{40pt} p{12pt} p{0.1pt} p{40pt} p{12pt}}
\hline
   \multirow{2}{*}{}&\multirow{2}{*}{ $N_x$} & \multicolumn{2}{c}{E}& & \multicolumn{2}{c}{H}& &\multicolumn{2}{c}{P} \\
\cline{3-4} \cline{6-7} \cline{9-10}
  &   & Error &Order &    & Error &Order &  & Error &Order\\
\hline
   \multirow{4}{*}{\ $\mathbb{P}^1$}
    &10 &1.394E-1  & --    &  & 5.100E-1  & --     & & 6.669E-2  &--    \\
    &20 &3.454E-2  & 2.013 &  & 1.218E-1  & 2.066  & & 1.662E-2  &2.005   \\
    &40 &8.628E-3  & 2.001 &  & 2.973E-2  & 2.035  & & 4.133E-3  &2.008   \\
    &80 &2.167E-3  & 1.993 &  & 7.341E-3  & 2.018  & & 1.028E-3  &2.007    \\
\hline
   \multirow{4}{*}{\ $\mathbb{P}^2$}
    &10 &2.879E-2  & --    &  & 5.601E-3  & --    & & 1.266E-2  &--    \\
    &20 &3.712E-3  & 2.955 &  & 7.000E-4  & 3.000 & & 1.626E-3  &2.961    \\
    &40 &4.745E-4  & 2.968 &  & 8.365E-5  & 3.065 & & 2.042E-4  &2.993   \\
    &80 &6.082E-5  & 2.964 &  & 1.043E-5  & 3.004 & & 2.423E-5  &3.075     \\
\hline
\end{tabular}
\caption{The $L^2$ errors and convergence orders of E,\,H,\,P for Example \ref{test1} with $\alpha=0.7$.}
\label{ex3alpha07}
\end{table}

In order to  test the efficiency of the proposed scheme, for comparison, we  also implement the following scheme \cite{LiHuangLinCOLE} to discretize  the fractional derivative 
\begin{eqnarray}\label{def:directscheme}
\frac{\partial^{\alpha}\,P(x,\,t_n)}{\partial t^{\alpha}}&=&
a_0\,P^n(x)+\sum_{m=1}^{n-1}\left(a_{n-m}-a_{n-m-1}\right)P^m(x)
+\mathcal{O}(\tau^{\frac{3}{2}}),\\
\nonumber a_m&=&\frac{\tau^{-\alpha}}{\Gamma(2-\alpha)}\left[ \left(m+1\right)^{1-\alpha}-m^{1-\alpha}\right],
\end{eqnarray}
where $\tau$ is the time step. We fix $\alpha=0.5$ and compare the  CPU time  between  our fast algorithm and the direct scheme (\ref{def:directscheme}). The linear  complexity   and  a significant  reduction of the CPU time by our fast algorithm are observed in Fig.\,\ref{fig:CPUT1}, where the number of spatial cells is fixed to be $Nx=10$ and the numbers of time steps are  $Nt=10000,\,20000,\,40000,\,80000,\,160000$.
\begin{figure}[htb]
\centering
\includegraphics[width=2.5in,height=2.0in]{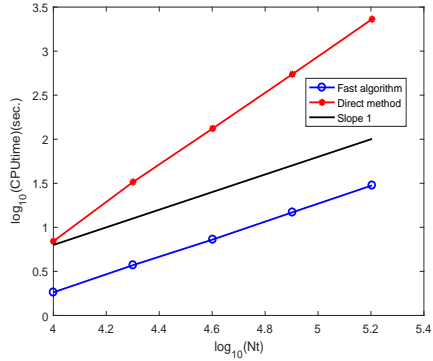}
\caption{The CPU time with fixed $Nx=10$ and different $Nt$ for Example \ref{test1}.}
\label{fig:CPUT1}
\end{figure}

\subsection{Energy analysis\label{test2}}
The purpose of  this example is  to illustrate the energy properties of the Cole-Cole model (\ref{eq:MaxColeColeH})-(\ref{eq:MaxColeColeP}).  To this end, we use $\mathbb{P}^1$ elements to solve the approximate system (\ref{eq:MaxColeColeHapp})-(\ref{eq:MaxColeColepsidapp})   with zero external force terms and periodic boundary condition and the following  initial  data
\begin{eqnarray*}
E_0(x)&=&\cos(\pi x)\sin(\pi x), \\
H_0(x)&=&2\pi\cos(\pi x)+\pi\sin(\pi x), \\
P_0(x)&=&0.
\end{eqnarray*}
The computation domain is  $\Omega=[0,\,2]$, which is  partitioned  by $800$ uniform elements.
The total simulation time is $T=2.5$ and the time step is set to be $\tau=h$.
The weights $\zeta_\ell$ and abscissae $\lambda_\ell$ are obtained by taking $L=20$ with frequency range $[0.5,\, 5]$.  Fig.\,\ref{fig:energyInitial} demonstrates the energy behaviors with respect to the evolution time, where different fractional orders of $\alpha=0.3,\,0.5,\,0.7$ are considered. Since the coefficients are constant in this example, the \emph{classical} energy $\tilde{\mathcal{E}}$ defined in \eqref{def:energyLi} and $\mathcal{E}_1(=\mathcal{E}_1^\sharp)$ differ only by a constant factor. Therefore the behavior of $\mathcal{E}_1^\sharp$ is the same as that of $\tilde{\mathcal{E}}$. Numerical results shown in Fig.\,\ref{fig:energyInitial}  are summarized below.

\begin{itemize}
\item As is indicated by the blue curves in Fig.\,\ref{fig:energyInitial} (d)-(f), the  energy  $\mathcal{E}^{\sharp}$ is  monotonically  decreasing in the sense that  $\mathcal{E}^{\sharp}(t_n)-\mathcal{E}^{\sharp}(t_{n+1})>0$ for all $n$. This is consistent with Theorem \ref{stableMaxwellCole}.

\item The nonlinear constraint optimization project preserves the positivity of the diffusive energy $\mathcal{E}^{\sharp}_2(t)$ given in (\ref{eq:energyE2primapp}) and the total energy   $\mathcal{E}^{\sharp}$ given in (\ref{eq:energyappTotalEp}), which can be  seen from the red and blue curves in Fig.\,\ref{fig:energyInitial} (a)-(c). Hence the approximate system  (\ref{eq:MaxColeColeHapp})-(\ref{eq:MaxColeColepsidapp}) is stable by Theorem \ref{stableMaxwellColeapp}.

\item As is demonstrated by the green curves in Fig.\,\ref{fig:energyInitial} (a)-(c), the classical energy function defined in (\ref{def:energy1}) satisfies $\widetilde{\mathcal{E}}(t_n)<\widetilde{\mathcal{E}}(t_0)$; this is consistent with Lemma 2.1 in \cite{LiHuangLinCOLE}. However, as can be observed from the green curves in Fig.\,\ref{fig:energyInitial} (d)-(f), the classical energy $\widetilde{\mathcal{E}}(t)$ is not a  monotonically decreasing function because  $\mathcal{E}_1(t_n)-\mathcal{E}_1(t_{n+1})<0$ for some $n$. This 'non-decreasing'  phenomenon  has also been   reported  in \cite{WangHuang} and \cite{bai2022second}.

\end{itemize}
\begin{figure}[htb]
\centering
\subfloat[$\alpha=0.3$]{%
   \includegraphics[width=2in,height=1.5in]{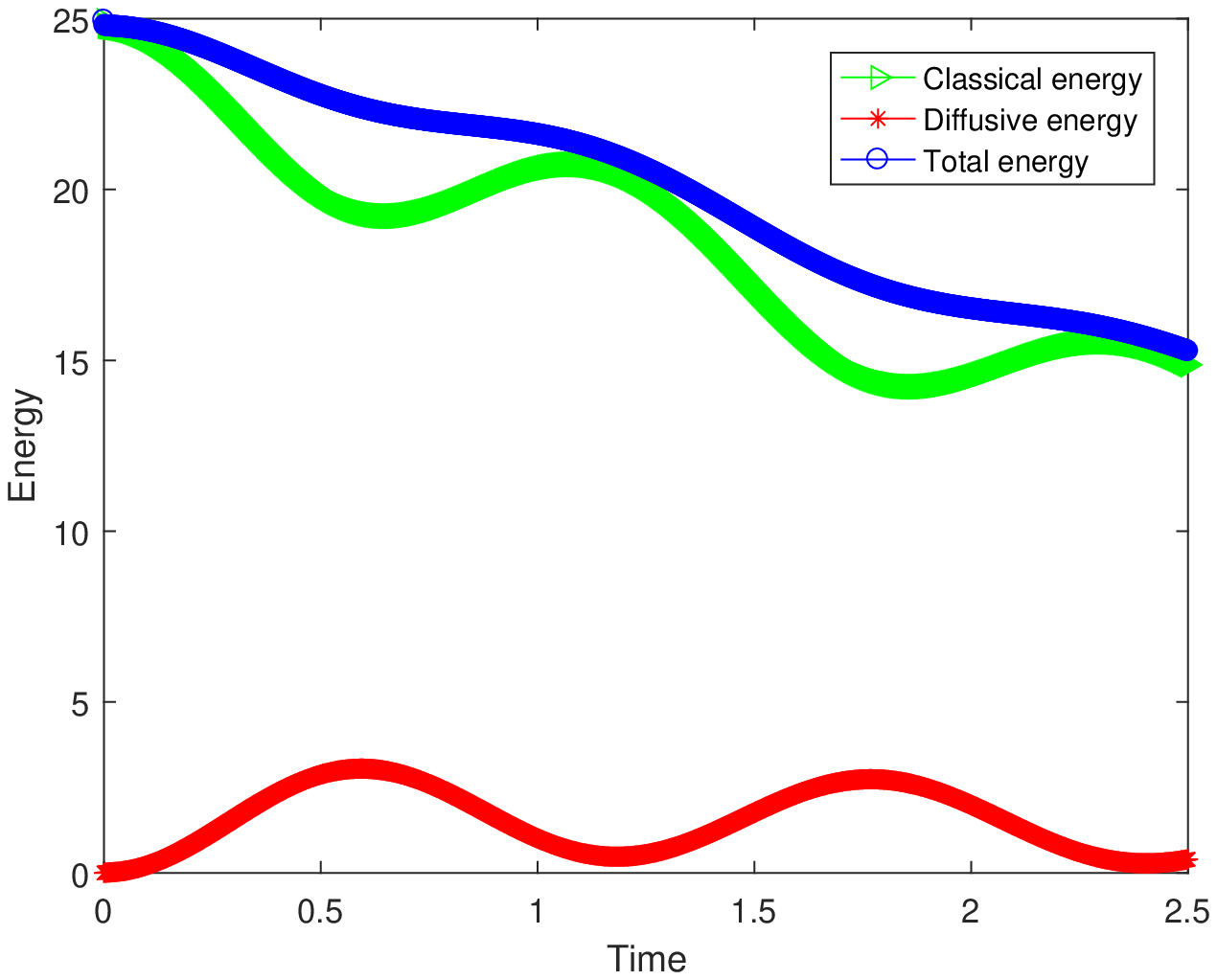}}
\subfloat[$\alpha=0.5$]{%
   \includegraphics[width=2in,height=1.5in]{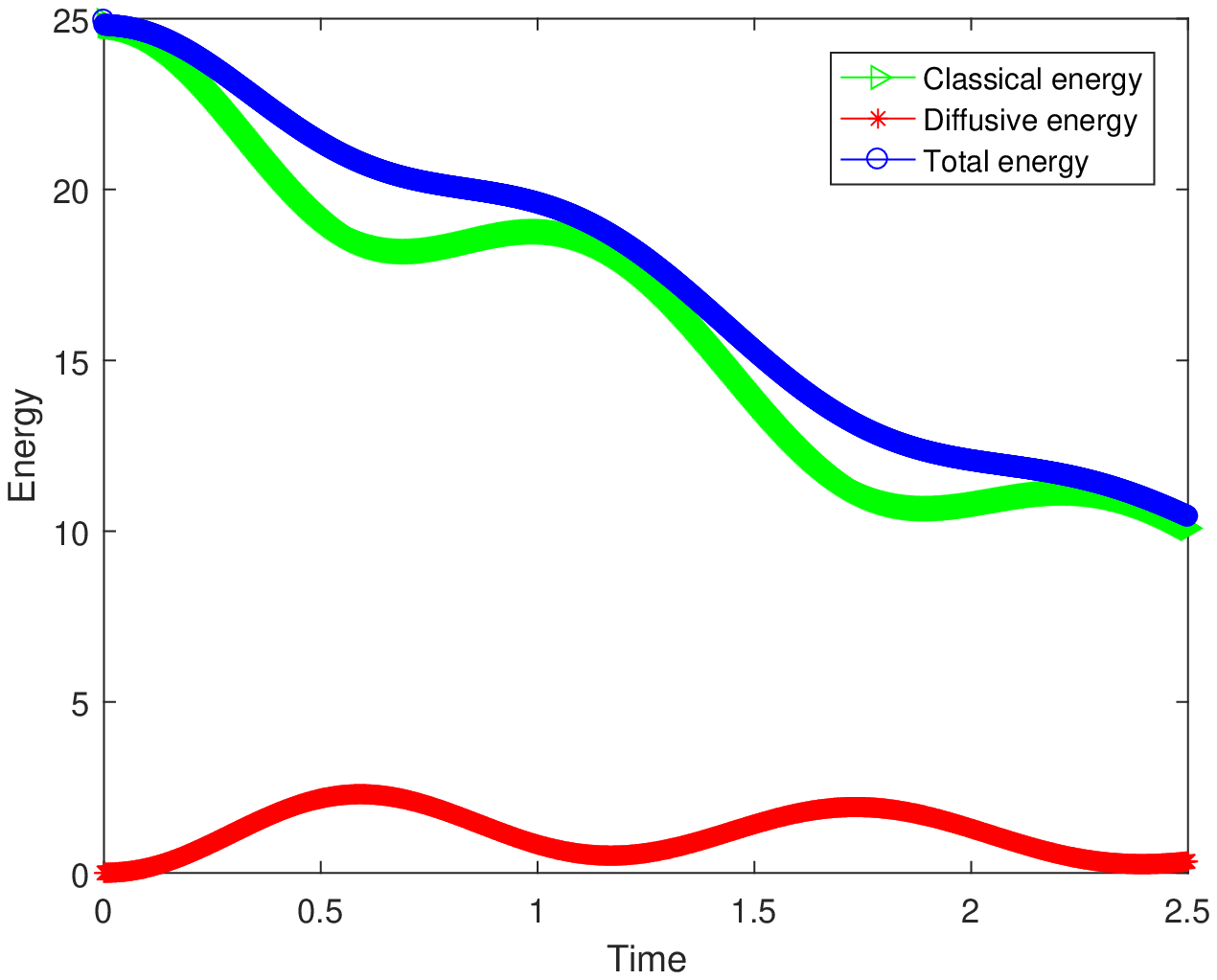}}
\subfloat[$\alpha=0.7$]{%
   \includegraphics[width=2in,height=1.5in]{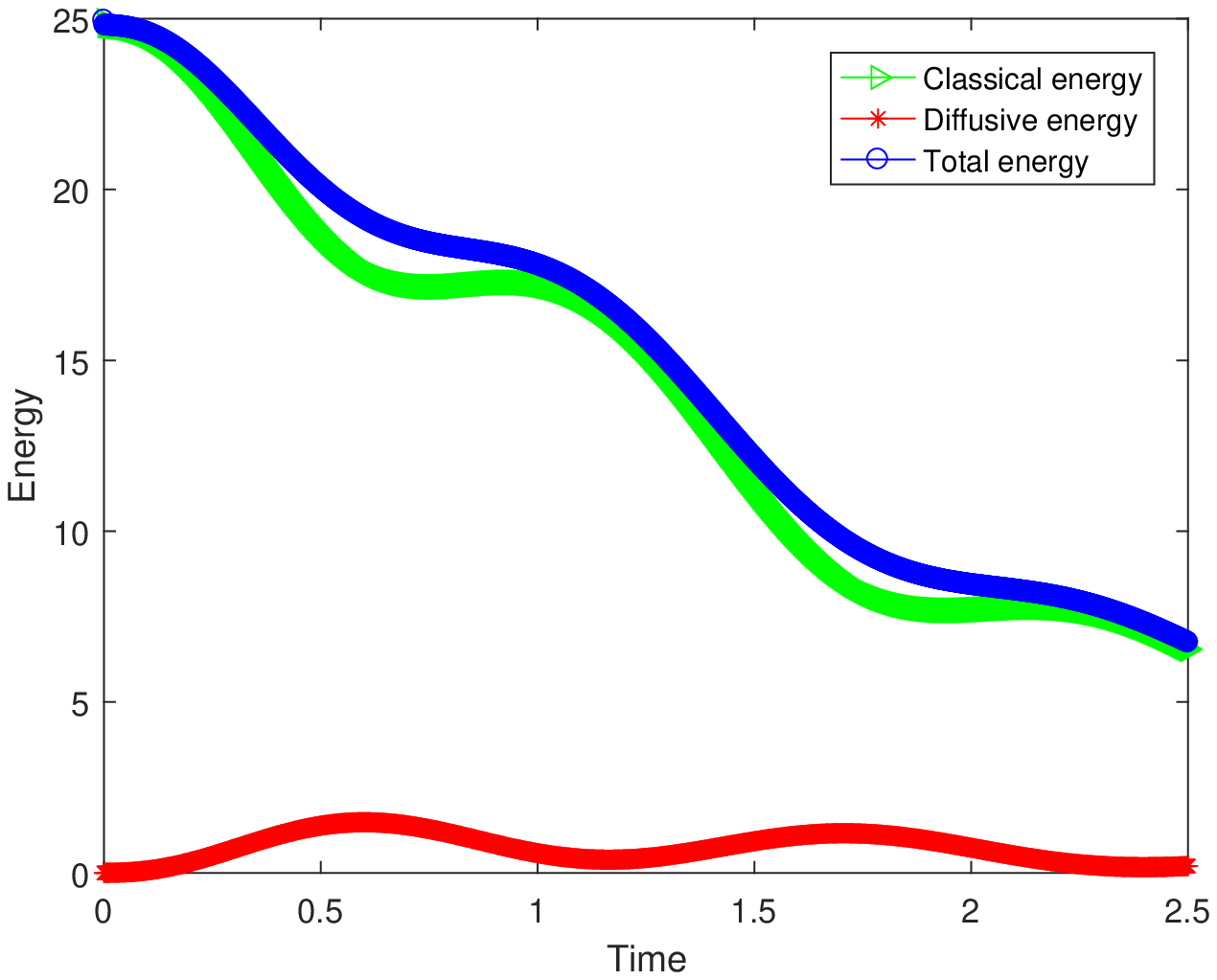}}\\
\subfloat[$\alpha=0.3$]{%
   \includegraphics[width=2in,height=1.5in]{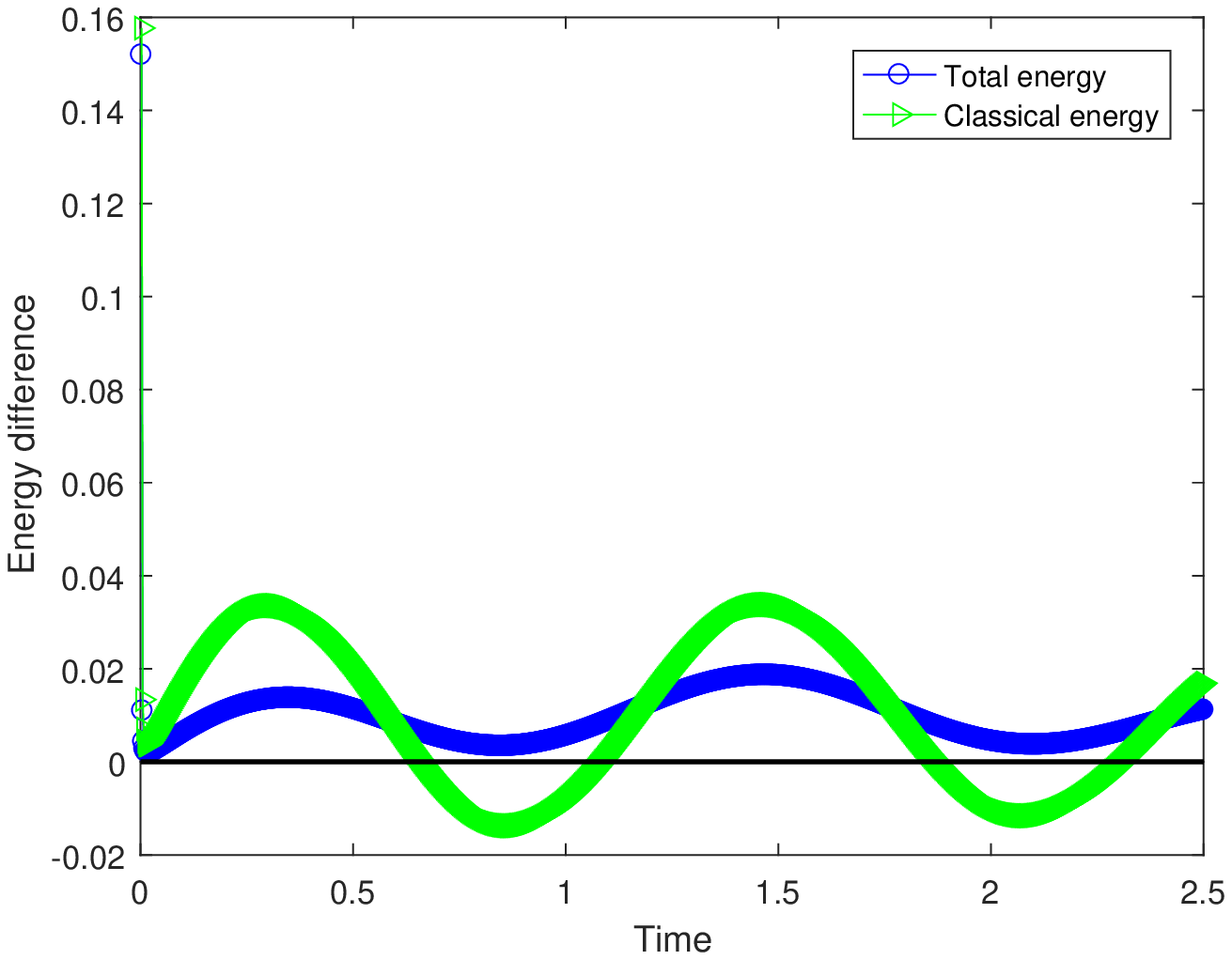}}
\subfloat[$\alpha=0.5$]{%
   \includegraphics[width=2in,height=1.5in]{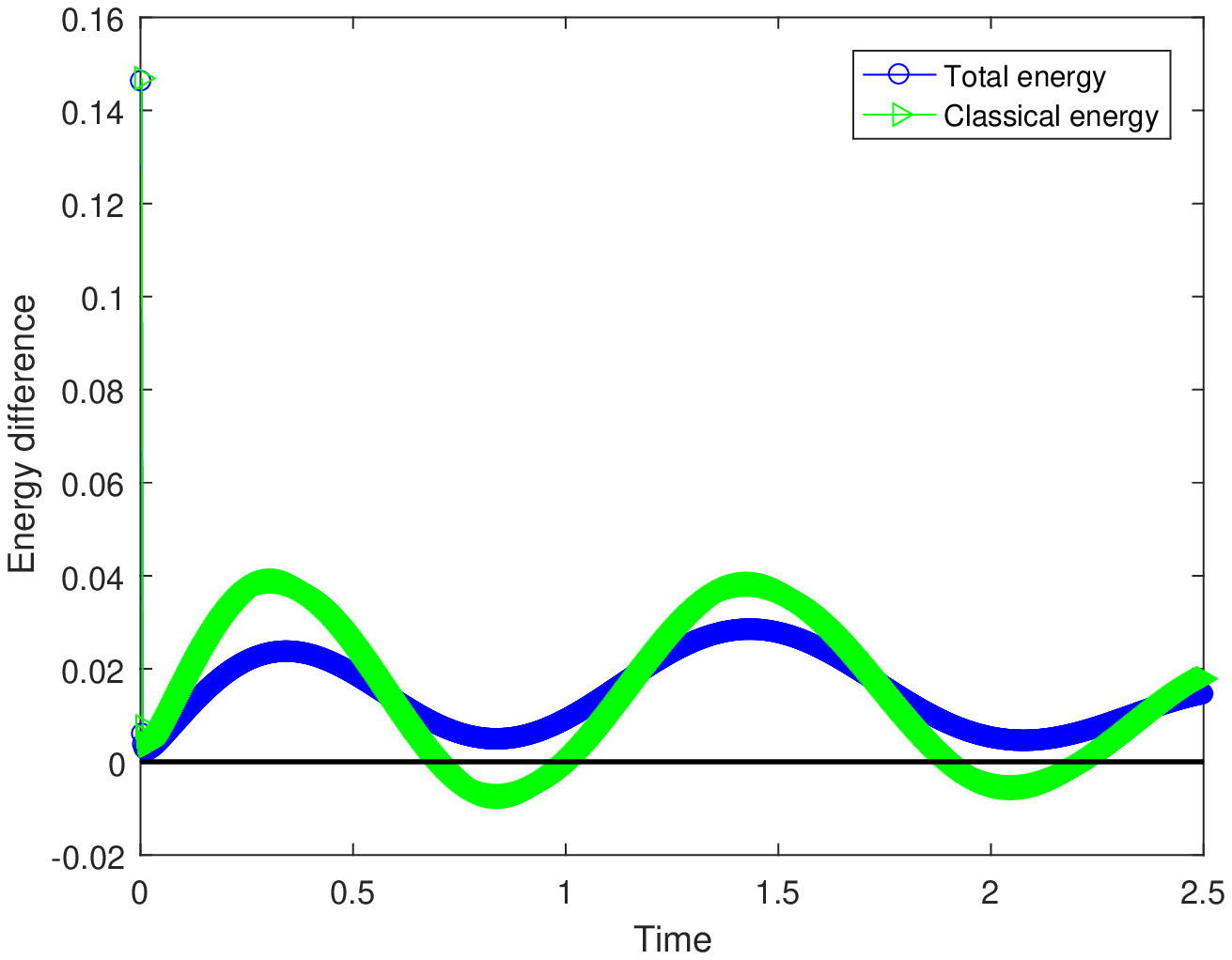}}
\subfloat[$\alpha=0.7$]{%
   \includegraphics[width=2in,height=1.5in]{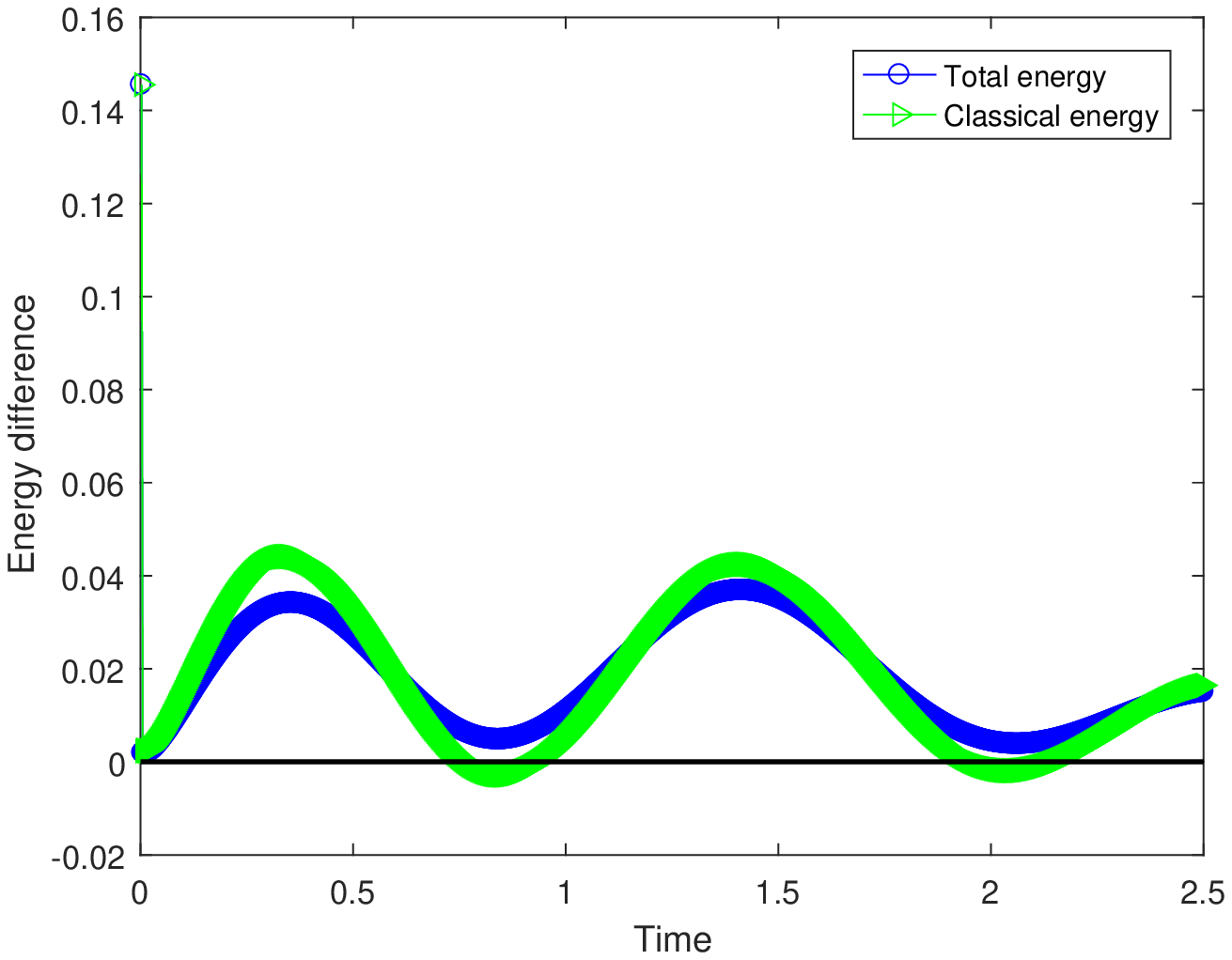}}
\caption{\label{enery_compare}The  energy behaviors of the Cole-Cole model with different fractional orders $\alpha$. In the legend, the {\it Classical\ energy}  refers to $\mathcal{E}_1$, the {\it Diffusive energy} to  $\mathcal{E}^{\sharp}_2$, the {\it Total energy}  to  $\mathcal{E}^{\sharp}$ and the {\it Energy difference}  refers to  ${\mathcal{F}}(t_n)-{\mathcal{F}}(t_{n+1})$ with ${\mathcal{F}}=\mathcal{E}^{
\sharp},\, \mathcal{E}_1$.}
\label{fig:energyInitial}
\end{figure}

\section{Conclusions}
In this work, we  develop a fast algorithm to solve the one-dimensional time-domain Maxwell's equations  for the Cole-Cole dispersive medium. A new, sharpened and monotonically decreasing  total energy function for the Cole-Cole model is derived    for  the first time, which   better   describes  the energy   of a Cole-Cole medium than the classical  energy  defined in \cite{LiHuangLinCOLE}. 

The  numerical challenge imposed by the  fractional derivative   involved in the polarization equation    is tackled by  using the diffusive representation, and  suitably chosen   quadrature formula.
A finite number of continuous auxiliary  variables   are introduced to approximate and localize the convolution kernel, which satisfy the local-in-time ODEs and can be easily solved.

The stability of the resulted approximate system is   proved to hold  as long as all the quadrature coefficients are positive. Therefore,  we establish a nonlinear constrained optimization  numerical scheme to preserve the positivity of the quadrature coefficients.

The spatial discretization is achieved by the DG method and we   presented a rigorous error analysis  for the  semi-discrete DG scheme for the constant coefficient case. Compared with the convergence rate $\mathcal{O}\left(h^{k+0.5}\right)$  in \cite{WangZhangZhang}, we  obtained an optimal convergence rate of $\mathcal{O}(h^{k+1})$ by a special choice of the numerical fluxes and projections. The  time discretization for the approximated system is achieved by the standard BDF2 formula and the
overall complexity for the proposed algorithm is $\mathcal{O}(N\,\text{log}N)$ with N being the total time steps.

\section*{Acknowledgments}
The work of J. Xie is partially supported by NSFC Grants (Nos. 12171274,\,12171465).
The work of M. Li is partially supported by a NSFC Grant (No. 11871139). 
The work of MYO's work was partial sponsored by US NSF Grant DMS-1821857.

\bibliographystyle{IMANUM-BIB}
\bibliography{IMANUM-refs}

\clearpage

\end{document}